\documentclass[11pt]{article}  
\usepackage{amsmath}
\usepackage{amsfonts}
\usepackage{amssymb}
\usepackage{amsthm}
\usepackage{euscript}
\usepackage{ifthen}
\usepackage{xifthen}
\usepackage{xcolor}
\usepackage{fullpage}
\usepackage[english]{babel}
\usepackage{mathtools}
\usepackage{hyperref}
\usepackage{authblk}
\hypersetup{colorlinks = true,linkcolor=black,citecolor=black}

\usepackage{algorithm} 
\usepackage{algorithmicx}
\usepackage[noend]{algpseudocode}

\usepackage{array}

\usepackage{imakeidx}
\makeindex

\usepackage{hyperref}
\hypersetup
{
	colorlinks,
	citecolor=black,
	filecolor=black,
	linkcolor=black,
	urlcolor=blue
}
\hypersetup{linktocpage}
\newtheorem{theorem}{Theorem}[section]
\newtheorem{lemma}{Lemma}[section]
\newtheorem{corollary}{Corollary}[section]
\numberwithin{equation}{section}
\newtheorem{assumption}[theorem]{Assumption}

\theoremstyle{definition}
 
\theoremstyle{remark}
\newtheorem{remark}[theorem]{Remark}

\newcommand{\brac}[1]{\left(#1\right)}

\newcommand{\brab}[1]{\left\{#1\right\}}
\newcommand{\set}[2]{\{#1\,:\,#2\}}
\newcommand{\black}[1]{{\color{black}{ #1 }}}

\newcommand{\norm}[2]{\left\|{#1}\right\|_{#2}}

\newcommand{\ba}{{\boldsymbol{a}}}
\newcommand{\bb}{{\boldsymbol{b}}}

\newcommand{\be}{{\boldsymbol{e}}}
\newcommand{\bj}{{\boldsymbol{j}}}
\newcommand{\bk}{{\boldsymbol{k}}}

\newcommand{\bs}{{\boldsymbol{s}}}

\newcommand{\bx}{{\boldsymbol{x}}}

\newcommand{\by}{{\boldsymbol{y}}}

\newcommand{\bz}{{\boldsymbol{z}}}

\newcommand{\brho}{{\boldsymbol{\rho}}}
\newcommand{\bvarrho}{{\boldsymbol{\varrho}}}
\newcommand{\bsigma}{{\boldsymbol{\sigma}}}

\newcommand{\balpha}{{\boldsymbol{\alpha}}}

\newcommand{\rd}{{\rm d}}

\def\ZZ{{\mathbb Z}}

\def\RR{{\mathbb R}}

\def\NN{{\mathbb N}}

\def\CC{{\mathbb C}}
\def\CC{{\mathbb C}}

\def\NN{{\mathbb N}}
\def\RR{{\mathbb R}}

\def\Vv{{\mathbb P}}

\def\Vv{{\mathcal P}}
\def\Kk{{\mathcal K}}
\def\FF{{\mathcal F}}
\def\Bb{{\mathcal B}}

\def\RRi{{\mathbb D}^\infty}

\def\RRi{{\mathbb R}^\infty}

\def\Aa{{\mathcal A}}
\def\Bb{{\mathcal B}}
\def\Cc{{\mathcal C}}

\def\Ii{{\mathcal I}}
\def\Jj{{\mathcal J}}
\def\Kk{{\mathcal K}}

\def\NN{{\mathcal N}}

\def\Pp{{\mathcal P}}

\def\Ss{{\mathcal S}}
\def\Tt{{\mathcal T}}
\def\Uu{{\mathcal U}}
\def\Vv{{\mathcal V}}
\def\Ww{{\mathcal W}}

\def\CC{{\mathbb C}}
\def\ZZ{{\mathbb Z}}
\def\NN{{\mathbb N}}
\def\RRR{{\mathbb R}}

\def\FF{{\mathbb F}}

\def\RRRi{{\mathbb R}^\infty}

\def\supp{\operatorname{supp}}
\def\dv{\operatorname{div}}

\def\div{\operatorname{div}}

\title{\sffamily  Simultaneous spatial-parametric collocation approximation for parametric PDEs with log-normal random inputs}

\author[*]{Dinh D\~ung}
\affil[*]{Information Technology Institute, Vietnam National University, Hanoi
	\protect\\
	144 Xuan Thuy, Cau Giay, Hanoi, Vietnam
	\protect\\
	Email: dinhzung@gmail.com}

\date{\today}
\begin{document}
\maketitle

\begin{abstract}
	We establish convergence rates for a fully discrete, multi-level, linear collocation method solving parametric elliptic PDEs on bounded polygonal domains with log-normal random inputs. The method uses a finite set of function evaluations in the spatial-parametric domain. Compared with the best-known fully discrete collocation rates, these rates are significantly improved and, up to logarithmic factors, match the rates of best $n$-term approximations.
	The results follow from applying general multi-level linear sampling recovery theory in abstract Bochner spaces -- via extended least-squares -- to infinite-dimensional holomorphic functions. The abstract multi-level recovery in Bochner spaces guarantees yield the improved rates when specialized to the parametric PDE setting.
	
	\medskip
	\noindent
	{\bf Keywords and Phrases}: Uncertainty quantification; Sampling recovery in Bochner spaces; Simultaneous spatial-parametric collocation approximation; Least squares approximation; Parametric PDEs with log-normal random inputs; Infinite dimensional holomorphic function; Kondrat'ev spaces; Convergence rate.
	
	\medskip
	\noindent
	{\bf Mathematics Subject Classifications (2020)}: 65C30, 65N15, 65N35, 41A25, 41A65. 	
\end{abstract}

\section{Introduction and main results}
\label{Introduction and main results}

Let $D \subset \RRR^2$ be a bounded  polygonal domain (recall that a  polygonal domain in $\RR^2$ is a polygon (which may have
precludes cusps and slits)  with a finite number of straight sides).  Consider the divergence-form diffusion elliptic equation 
\begin{equation} \label{ellip}
	- \dv (a\nabla u)
	\ = \
	f \quad \text{in} \quad D,
	\quad u|_{\partial D} \ = \ 0, 
\end{equation}
for  a given fixed right-hand side $f$ and a 
spatially variable scalar diffusion coefficient $a$.
Denote by $V:= H^1_0(D)$ the energy space and $V' = H^{-1}(D)$ the dual space of $V$. Assume that  $f \in H^{-1}(D)$ and  $a \in L_\infty(D)$ (in what follows this preliminary assumption always holds without mention). If $a$ satisfies the ellipticity assumption
\begin{equation} \nonumber
	0<a_{\min} \leq a \leq a_{\max}<\infty,
\end{equation}
by the well-known Lax-Milgram lemma, there exists a unique weak 
solution $u \in V$  to the equation~\eqref{ellip}  satisfying the variational equation
\begin{equation} \nonumber
	\int_{D} a\nabla u \cdot \nabla v \, \rd \bx
	\ = \
	\langle f , v \rangle,  \quad \forall v \in V.
\end{equation}
For the equation~\eqref{ellip}, 
we consider  the diffusion coefficients having a parametric form $a=a(\by)$, where $\by=(y_j)_{j \in \NN}$
is a sequence of real-valued parameters ranging in the set 
$\RRRi$.
Denote by $u(\by)$ the weak solution to the 
parametric  diffusion divergence-form elliptic equation 
\begin{equation} \label{parametricPDE}
	- {\rm div} (a(\by)\nabla u(\by))
	\ = \
	f \quad \text{in} \quad D,
	\quad u(\by)|_{\partial D} \ = \ 0. 
\end{equation}	
The present paper considers the case \black{of parametric equation \eqref{parametricPDE}} when the \black{random} diffusion coefficient $a$  is of the log-normal form
\begin{equation} \label{lognormal}
	a(\by)=\exp(b(\by)), \quad {\text{with }}\ b(\by)=\sum_{j = 1}^\infty y_j\psi_j,
\end{equation}
where $y_j$ are i.i.d. standard Gaussian random 
variables and 
 $\psi_j \in L_\infty(D)$.

For $r\in\NN_0$ and $\varkappa\in\RR$,  the Kondrat'ev spaces
 $\Kk^{r}_{\varkappa}(D)$ and $\Ww^{r}_\infty(D)$ are defined as  the weighted normed spaces of functions $\black{v}$ on $D$ equipped with norms
\begin{equation*}
	\|\black{v}\|_{ \Kk^r_\varkappa(D)}
	:= \
	\sum_{|\balpha|\leq r}\|\tau_D^{|\balpha|-\varkappa}D^\balpha \black{v}\|_{L^2(D)}
	\qquad\text{and}\qquad
	\|\black{v}\|_{ \Ww^r_\infty(D)}
	:= \
	\sum_{|\balpha|\leq r}\|\tau_D^{|\balpha|}D^\balpha \black{v}\|_{L^\infty(D)},
\end{equation*}
respectively. Here, 
$\tau_{D}:D\to [0,1]$ denotes a fixed smooth function that coincides
with the distance to the nearest corner, in a neighborhood of each
corner, $D^\balpha$ denotes the
weak partial derivative  of order $\balpha\in \NN_0^2$,  and $|\balpha|:= \alpha_1 +\alpha_2$.
The function spaces $\Kk^r_\varkappa(D)$ and
$\Ww^r_\infty(D)$ endowed with these norms are Banach spaces,
and $\Kk^r_\varkappa(D)$ are separable Hilbert spaces.  
An embedding of these spaces is
$ \Kk^1_1(D) \hookrightarrow H^1_0(D)$ \cite[Lemma 1.5]{BNZ2005}, i.e., there exits a positive constant $C$ such that
\begin{equation} \label{EmbeddingInequality}
	\|v\|_{H^1_0(D)}
	\ \le \ 	
	C \|v\|_{\Kk^1_1(D)}, \ \  v \in \Kk^1_1(D).
\end{equation}

%

By \cite[Theorem 3.29]{DNSZ2023}, for $r\ge 2$,
$f\in \Kk^{r-2}_{\varkappa-1}(D)$ and $a\in \Ww^{r-1}_\infty(D)$, the
weak solution $\black{u}$ to \eqref{ellip} belongs to
$\Kk_{\varkappa+1}^{r}(D)$ provided that with
\begin{equation*}
	\rho(a) := \underset{\bx\in D}{\operatorname{ess\,inf}}\, a(\bx)>0,
	\ \ \ \text{and} \ \ \
	|\varkappa|<\frac{\rho(a)}{\nu \norm{a}{L_\infty(D)}},
\end{equation*}
where $\nu$ is a constant depending on $D$ and $r$.

In computational uncertainty quantification, the problem of efficient   approximation for  parametric and stochastic PDEs has been of great interest and achieved  significant progress in recent years. Depending on a particular setting,  as usual, this problem  leads to an approximation problem in a Bochner space $L_2(U,X;\mu)$ with an appropriate separable Hilbert space $X$, an infinite-dimensional domain $U$ and a probability measure $\mu$ on $U$, where parametric solutions $u(\by)$, $\by \in U$, to parametric and stochastic PDEs, are treated as elements of  $L_2(U,X;\mu)$ and $U$ the parametric domain.   There is a  vast number of works on this topic to not mention all of them.  We  point out just some works \cite{ABDM2024,ADM2024,BNT2007,BCDC17,BCDM17,BCM17,BD2024,BTNT12,CCS15,CoDe15a,CCS13,CCMNT2015,CDS10,CDS11,Dung19,Dung21,DNSZ2023, EST18,HoSc14,MNST2014,NTW2008,NTW2008a,ZDS19,ZS20} which are directly related to the  problem setting in our paper.  In semi-discrete (intrusive and non-intrusive) approximations, \cite{ADM2024,BNT2007,BCDM17,BCM17,BTNT12,CCS15,CoDe15a,CCS13,CCMNT2015,Dung21,DD-Erratum23,DNSZ2023,EST18,MNST2014,NTW2008,NTW2008a,ZS20} discretizations are processed only with respect to the parametric variables, but not spatial variables, the approximants still belong to the infinite dimensional spatial space and hence  are not useful for practical applications. 
They  should themselves be approximated by spatial discretizations by means of finite elements, wavelets  or spectral Galerkin methods, etc.. The problem of fully discrete (and multilevel)  approximation of parametric PDEs with  random inputs and of relevant infinite-dimensional holomorphic functions has been investigated  in the works
\cite{BCDC17,CDS10,CDS11,Dung19,Dung21,DD-Erratum23,DNSZ2023,HoSc14,ZDS19} based on some spatial approximation properties and parametric summabilities of the GPC expansion coefficients of solutions.  Some bounds for convergence rates of spectral and collocation fully discrete approximation have been proven in these papers.  Observe that by the problem setting a convergence rate of any fully discrete approximation  is not better than the convergence rate given by the spatial approximation properties. Moreover, it was proven that they coincide  in the case of sufficiently small spatial regularity. In the case of higher spatial regularity, the known bounds of convergence rates of fully discrete approximation are less than this regularity and depend essentially on parametric summability properties  of the generalized polynomial chaos (GPC) expansion coefficients of solutions (see the above cited papers for detail). 

A natural question arising is whether the known convergence rates are optimal and whether one can improve them. The aim of the present paper is to improve these known bounds in the case of higher spatial regularity of the solution of the parametric equation \eqref{parametricPDE} with log-normal inputs \eqref{lognormal}.  
More precisely,
we are interested in the problem of  fully discrete  simultaneous spatial-parametric linear  collocation approximation of solutions $u$ to parametric PDEs \eqref{parametricPDE} on bounded polygonal domain $D$ with log-normal random inputs  \eqref{lognormal} and its convergence rate, based on a finite number of its particular spatial-parametric values $u(\bx_1,\by_1),..., u(\bx_n,\by_n)$. The approximation error is measured by the norm of   the Bochner space \black{$L_2(\RRi,V;\gamma)$}  associated with the energy space $V$ and the infinite standard Gaussian measure $\gamma$ on $\RRi$ (see Subsection~\ref{Extended least squares sampling algorithms in Bochner spaces} for definition of Bochner space).	The present paper can be considered as  a continuation of the works 
\cite{BD2024, Dung21,DNSZ2023} (see also \cite{Dung2025-v14}, the arXiv-updated version of \cite{Dung21} with minor corrections).  In \cite{Dung21,DNSZ2023}, we have investigated the problem of fully discrete multi-level collocation approximation  of stochastic and parametric elliptic PDEs with log-normal inputs by using sparse-grid GPC Lagrange interpolation at the zeros of Hermite  polynomials.  In the recent paper \cite{BD2024}, we have investigated semi-discrete  approximation of these equations by employing extended least squares sampling algorithms in Bochner spaces.  In the present paper, we develop and combine the different techniques used in  these papers for solving the formulated problem.

We give a short description of main contribution of the present paper with some comments.

Let	
$r\in\NN$, $r \ge  2$, $f\in \Kk_{\varkappa-1}^{r-2}(D)$ and
$\psi_k\in W^{1}_\infty(D)\cap\Ww_{\infty}^{r-1}(D)$, $k \in \NN$. 
For $i=1,2$, let the sequences $\bb_i:=(b_{i,j})_{j\in\NN}$ be defined by
\begin{equation*}
	b_{1;j}:=\norm{\psi_j}{\black{L_\infty(D)}},\ \ \text{and} \ \
	b_{2;j}:=\max\big\{\norm{\psi_j}{W^{1}_\infty(D)},\norm{\psi_j}{\Ww_\infty^{r-1}(D)}\big\},
\end{equation*}
 and 
satisfy the condition 
$\bb_i\in\ell^{p_i}(\NN)$, respectively, with $0< p_1\le  p_2 < 1$. 
	Let the numbers $\alpha$ and  $\beta$ be defined by
		\begin{equation} \label{alpha,beta}
		\alpha:= \frac{r-1}{2}, \quad	\beta := \brac{\frac 1 {p_1} - \black{1}} \frac{\alpha}{\alpha + \delta}, \quad 
		\delta := \frac 1 {p_1} - \frac 1 {p_2}.
	\end{equation}	

	\black{For any fixed positive number $\varepsilon$,  there exist positive constants $C$ and $C_\varepsilon$ such that for all $n \ge 2$, there exist  points $(\bx_1,\by_1),...,(\bx_n,\by_n) \in D \times \RRi$ and functions 
	$\varphi_1,...,\varphi_n \in V$ and  $h_1,...,h_n \in L_2(\RRi,\RR;\gamma)$  such that for the linear sampling algorithm $S_n$ on the spatial-parametric domain $D\times \RRi$ defined  by
	\begin{equation*}
		S_n(v)(\bx,\by): = \sum_{i=1}^n  v(\bx_i,\by_i) \varphi_i (\bx) h_i(\by),  
		\ \ \bx \in D, \ \ \by \in \RRi,
	\end{equation*}	
	and for the parametric solution $u$ to the equation \eqref{parametricPDE} with log-normal random inputs \eqref{lognormal},
	it holds true the error bounds}
	 \begin{equation} 	\label{main-result}
		\begin{split}
			\big\|u-S_n u\big\|_{L_2(\RRi,V;\gamma)} 
			\ \le \
				\begin{cases}
				C	n^{-\alpha}\quad &{\rm if }  \ \alpha \le \black{1/p_2 - 3/2},\\
				C_\varepsilon	n^{-\alpha}(\log n )^{\alpha + \black{1/2}+\varepsilon}
				\quad &{\rm if }  \ \black{1/p_2 - 3/2 <  \alpha < 1/p_2 - 1},\\
				C_\varepsilon	n^{-\alpha}(\log n )^{\alpha + \black{3/2} +\varepsilon}\quad &{\rm if }  \ \alpha = \black{1/p_2 -1},\\
				\black{C_\varepsilon}	n^{-\beta}
				(\log n )^{\black{1 + \beta (1 + 1/2 \alpha )+ \varepsilon\beta/\alpha}} \quad  &{\rm if }  \ \alpha > \black{1/p_2 - 1},
			\end{cases}			
		\end{split}
	\end{equation}
	
	\black{The convergence rates in \eqref{main-result}   significantly improve the best-known convergence rates \cite[Theorem 7.5 and page 177]{DNSZ2023} of  fully discrete multilevel Smolyak sparse-grid interpolation. Moreover, in contrast to the previous works, in the particular intermediate case 
		$1/p_2 - 3/2< \alpha \le 1/p_2 -\black{1}$, the maximal convergence rate $n^{-\alpha}$ still holds with a logarithm factor.	} 
It is important to highlight that  with some logarithm factors, the convergence rates in \eqref{main-result} coincide with the convergence rates $n^{-\min(\alpha,\beta)}$ of  fully discrete best $n$-term approximation of the parametric solution $u$ in the space  $L_2(\RRi,V;\gamma)$. 
	\black{See Remark \ref{remark} for a detailed commentary on the related results in prior works and contrast of them  with  the main findings \eqref{main-result} of this paper.}

\black{It is worth emphasizing}  that $S_n$ is a fully discrete multi-level collocation  approximation method.	At each level  in $S_n$, the spatial component  is based on finite element Lagrange interpolation associated with triangulations of the spatial domain $D$, while  the parametric component is based on Hermite-Lagrange GPC interpolation in the case $\alpha \le \black{1/p_2 - 3/2}$, and on  extended least squares sampling algorithms in the cases $\alpha > \black{1/p_2 - 3/2}$. 

 Following the approach in \cite{BD2024, Dung21,DNSZ2023}, the convergence rate results in \eqref{main-result} and constructions of the linear sampling algorithms  $S_n$ realizing them, are obtained as consequences of applications to infinite-dimensional holomorphic functions of general results on multi-level linear sampling recovery in abstract Bochner spaces. In the case of small spatial regularity  \black{$1/p_2 - 3/2$}, we use a modification of methods of  sparse-grid Hermite-Lagrange GPC interpolation in Bochner spaces in \cite{Dung21}. In the cases of higher spatial regularity 
 \black{$\alpha > 1/p_2 - 3/2$}, we used
extended least squares sampling algorithms in Bochner spaces and  some recent  results on linear sampling recovery of functions in reproducing kernel Hilbert spaces \cite{BSU23, DKU2023,KUV21,KU21a}. We would like to emphasize that the general theory of sampling recovery in abstract Bochner spaces presented  in this paper as well as  \cite{BD2024, Dung21,DNSZ2023} is applicable to uncertainty quantification for a wide class of parametric and stochastic PDEs.

This  paper is organized as follows. 

In Section \ref{Sampling recovery in Bochner spaces}, we investigate   multi-level linear sampling recovery  of functions in an abstract Bochner space $L_2(U,X;\mu)$ for a Hilbert space $X$ and a probability measure space \black{$(U,\Aa,\mu)$}, based on some weighted $\ell_2$-summability of their coefficients of an orthonormal expansion, and some approximation properties in the space $X$. The received results are then applied to $L_2(\RRi,X;\gamma)$, the Bochner space associated with a Hilbert space $X$ and the infinite standard Gaussian measure $\gamma$ on $\RRi$. 
In Section \ref{Applications to  holomorphic functions}, we apply the results on  multi-level linear sampling recovery in the space $L_2(\RRi,X;\gamma)$ in the previous section to $(\bb,\xi,\delta,X)$-holomorphic functions on $\RRi$.
In Section \ref{Multi-level sparse-grid interpolation algorithms},
we construct fully discrete multi-level sparse-grid  sampling algorithms
for  $(\bb,\xi,\delta,X)$-holomorphic functions, based on GPC Lagrange-Hermite interpolation, and prove convergence rates of the approximation by them. 
In Section \ref{Applications to  parametric PDEs}, 
 we prove that the parametric solution $u$ to
the parametric elliptic PDEs \eqref{ellip} on bounded polygonal domain with log-normal inputs \eqref{lognormal} is a $(\bb_j,\xi,\delta,V)$-holomorphic function. This allows to prove convergence rates of 
a fully discrete multi-level collocation  algorithm by applying the results in Sections \ref{Applications to  holomorphic functions} and \ref{Multi-level sparse-grid interpolation algorithms}.
In Section \ref{Extensions},
we  discuss various  least squares sampling algorithms  for functions in the reproducing kernel Hilbert space, and inequalities between sampling $n$-widths and Kolmogorov $n$-widths of the unit ball of this space, and how to apply these inequalities to obtain corresponding convergence rates of  multi-level linear sampling recovery in abstract Bochner spaces and  of fully discrete multi-level collocation  approximation of  the parametric solution $u$ to
the parametric elliptic PDEs \eqref{ellip} on bounded polygonal domain with log-normal inputs \eqref{lognormal}.

\medskip
\noindent
{\bf Notation} \  As usual, $\NN$ denotes the natural numbers, $\ZZ$  the integers, $\RR$ the real numbers, $\CC$ the complex numbers,  and 
$ \NN_0:= \{s \in \ZZ: s \ge 0 \}$.
We denote  $\RR^\infty$ the
set of all sequences $\by = (y_j)_{j\in \NN}$ with $y_j\in \RR$.
 Denote by $\FF$  the set of all sequences of non-negative integers $\bs=(s_j)_{j \in \NN}$ such that their support $\supp (\bs):= \{j \in \NN: s_j >0\}$ is a finite set.  
If  $\ba= (a_j)_{j \in \Jj}$  is a set of positive numbers with  any index set $\Jj$, then we use the notation 
$\ba^{-1}:= (a_j^{-1})_{j \in \Jj}$.
We use letters $C$  and $K$ to denote general 
positive constants which may take different values, and $C_{a,b,...,d}$  and $K_{a,b,...,d}$ constants depending on $a,b,...,d$.
Denote by $|G|$ the cardinality of the set $G$.


\section{Sampling recovery in Bochner spaces}
\label{Sampling recovery in Bochner spaces}	
 
In this section, we investigate    multi-level linear sampling recovery  of functions in  abstract Bochner spaces $L_2(U,X;\mu)$ for a Hilbert space $X$ and a probability measure \black{$\mu$}, based on some weighted $\ell_2$-summability of their coefficients of an orthonormal expansion, and some approximation properties in the space $X$. We develop  the methods used in \cite{BD2024, Dung21,DNSZ2023} to construct   fully discrete (multi-level) sampling algorithms by using the approximation properties in the space $X$ combined with extended least squares sampling algorithms at each level, and prove convergence rates of approximation by them.  The received results are then applied to $L_2(\RRi,X;\gamma)$, the Bochner space associated with a Hilbert space $X$ and the infinite standard Gaussian measure $\gamma$ on $\RRi$.

 \subsection{Function spaces}
 \label{Function spaces}
 \black{Let $(U,\Aa,\mu)$ be a probability measure space with $U$ being a separable topological space and let $X$ be a separable complex  Hilbert space.
    Denote by $L_2(U,X;\mu)$  the Bochner space  of  strongly $\mu$-measurable mappings $v$ from $U$ to $X$, equipped with the norm
 \begin{equation} \label{BochnerSpaceNorm}
 	\|v\|_{L_2(U,X;\mu)}
 	:= \
 	\left(\int_{U} \|v(\by)\|_X^2 \, \rd \mu(\by) \right)^{1/2}.
 \end{equation}
 Denote by $L_2(U,\CC;\mu)$ the Hilbert space which consists of all square $\mu$-integrable complex functions on $U$.
 Since $U$ is separable, the spaces $L_2(U,\CC;\mu)$ and $L_2(U,X;\mu)$ are also  separable complex Hilbert spaces,  and moreover,}
 $$
 L_2(U,X;\mu) = L_2(U,\CC;\mu) \otimes X.
 $$
  
 Let $\brac{\varphi_s}_{s \in \NN}$ be a fixed orthonormal basis of  $L_2(U,\CC;\mu)$.  Then every  function $v \in L_2(U,X;\mu)$  can be represented  by the expansion
 \begin{equation} \label{series}
 	v(\by)=\sum_{s \in \NN} v_s \,\varphi_s(\by), \quad v_s \in X,
 \end{equation}
 with the series convergence in $L_2(U,X;\mu)$,	
 where
 \begin{equation*}
 	v_s:=\int_U v(\by)\,\overline{\varphi_s(\by)}\, \rd\mu (\by), \quad s \in \NN.
 	\label{hermite}
 \end{equation*}
 Moreover, for every $v \in L_2(U,X;\mu)$  represented by the 
 series \eqref{series},  Parseval's identity holds:
 \begin{equation} \nonumber
 	\|v\|_{L_2(U,X; \mu)}^2
 	\ = \ \sum_{s\in \NN} \|v_s\|_X^2.
 \end{equation}

 Let $\bsigma=\brac{\sigma_{s}}_{s \in \NN}$ be  a  \black{non-decreasing} sequence of positive numbers strictly larger than $1$ such that 
 $\bsigma^{-1}:=\brac{\sigma_{s}^{-1}}_{s \in \NN} \in \ell_2(\NN)$. 
 For given $U$ and $\mu$, denote by $H_{X,\bsigma}$ the \black{ subspace} of all functions 
 $v \in L_2(U,X;\mu)$ such that the norm
 \begin{equation} \nonumber
 	\|v\|_{H_{X,\bsigma}}
 	:= \
 	\brac{\sum_{s \in \NN} \brac{\sigma_{s} \|v_s\|_X}^2}^{1/2} < \infty.		
 \end{equation}
 \black{Similarly,} the space $H_{\CC,\bsigma}$ is the \black{ subspace} in $L_2(U,\CC;\mu)$ equipped with its own inner product
 \begin{equation} \nonumber
 	\langle f,g \rangle_{H_{\CC,\bsigma}}
 	:= \
 	\sum_{s \in \NN} \sigma_{s}^2 
 	\langle f,\varphi_s \rangle_{L_2(U,\CC;\mu)}
 	\overline{\langle g,\varphi_s \rangle_{L_2(U,\CC;\mu)}}.
 \end{equation}
 The space $H_{\CC,\bsigma}$ is a reproducing kernel Hilbert space with the reproducing kernel
 \begin{equation} \nonumber
 	K(\cdot,\by)
 	:= \
 	\sum_{s \in \NN} \sigma_{s}^{-2} \varphi_s(\cdot)\overline{\varphi_s(\by)}
 \end{equation}
 having the eigenfunctions $\brac{\varphi_s}_{s \in \NN}$ and the eigenvalues 
 \black{$\brac{\sigma_s}_{s \in \NN}$}. Moreover, $K(\bx,\by)$ satisfies the finite trace assumption
 \begin{equation*}\label{trace assumption}
 	\int_U K(\by,\by) \rd \mu(\by) \ < \ \infty.
 \end{equation*}
 
 From the separability of  \black{$H_{\CC,\bsigma}$} it follows that there exists  a set $U_0 \subset U$ satisfying  $\mu(U\setminus U_0)=0$
 such that
 \begin{equation} \nonumber
 	K(\bx,\by)
 	:= \
 	\sum_{s \in \NN} \sigma_s^{-2} \varphi_s(\bx)\overline{\varphi_s(\by)}, \ \ \forall \bx, \by \in U_0,
 \end{equation}
 and
 \begin{equation} \nonumber
 	f(\by)
 	\ = \
 	\sum_{s \in \NN} 
 	\sigma_s^{-2}
 	\langle f,\varphi_s \rangle_{H_{\CC,\bsigma}}
 	\varphi_s(\by), \ \ 
 	\forall  f \in  L_2(U,\CC;\mu),  \ \forall  \by \in U_0.
 \end{equation}
 This means that the pointwise evaluations $f(\by)$ are well-defined for every $\by$ belonging the full measure subset $U_0$ of $U$ (cf. \cite{DKU2023}). 
 Let $\brac{\psi_k}_{k \in \NN}$  be an orthonormal basis of  $X$. 
 If $v \in H_{X,\bsigma}$, then $\langle v,\psi_k \rangle_X \in H_{\CC,\bsigma}$ for every $k$. 
 Consequently, the pointwise evaluations $\langle v(\by),\psi_k \rangle_X $ are well-defined for every
 $\by \in U_0$. This implies that the pointwise evaluations  $v(\by)$ are also well-defined for every
 $\by \in U_0$.
 
 Let $(V_k)_{\black{k \in \NN}}$ be a given sequence of subspaces 
 $V_k \subset X$ of dimension $2^{\black{k-1}}$, and $(P_k^X)_{\black{k \in \NN}}$  a given sequence of uniformly bounded linear projectors  $P_k^X$ from $X$ onto $V_k$. We extend the operators $P_k^X$ as  uniformly bounded linear operators in $L_2(U,X;\mu)$ (which with an abuse is denoted again by $P_k^X$)  such that
 \begin{equation} \label{(P_k v)}
 (P_k^X v)(\by):= P_k^X (v(\by)), \ \ \by \in U_0.
  \end{equation}
For $i=1,2$, let   $0< q_i \le 2$ and $\bsigma_i:=(\sigma_{i;s})_{s \in \NN}$  be given \black{non-decreasing} sequences of numbers strictly larger than $1$, such that $\bsigma_i^{-1}:=\brac{\sigma_{i;s}^{-1}}_{s \in \NN} \in \ell_{q_i}(\NN)$. 
For  a given number $\alpha> 0$, denote by $ H_X^{\alpha}$ the linear subspace in $L_2(U,X;\mu)$ of all $v$ such that the norm
 \begin{equation} \label{H_{X}^{alpha}}
 	\|v\|_{H_{X}^{\alpha}}
 	:= \
 	\brac{\|v\|_{H_{X,\bsigma_1}}^2 + \|v\|_{H_{X,\bsigma_2}^{\alpha}}^2}^{1/2} < \infty, 
 \end{equation}
 where the semi-norm
 $\norm{v}{H_{X,\bsigma_2}^{\alpha}}$ is defined by 
 \begin{equation} \label{H_{X}^{alpha}semi-norm}
 	\|v\|_{H_{X,\bsigma_2}^{\alpha}}
 	:= \
 	\sup_{k \in \NN}	2^{\alpha k} \norm{v - P_k^X v} {H_{X,\bsigma_2}}.
 \end{equation}
 
 Similarly, for given numbers $\alpha> 0$ and $\tau > 0$, denote by $ H_X^{\alpha,\tau}$ the linear subspace in $L_2(U,X;\mu)$ of all $v$ such that the norm
 \begin{equation} \label{H_{X}^{alpha,tau}}
 	\|v\|_{H_{X}^{\alpha,\tau}}
 	:= \
 	\brac{\|v\|_{H_{X,\bsigma_1}}^2 + \|v\|_{H_{X,\bsigma_2}^{\alpha,\tau}}^2}^{1/2} < \infty, 
 \end{equation}
 where the semi-norm
 $\norm{v}{H_{X,\bsigma_2}^{\alpha,\tau}}$ is defined by 
 \begin{equation} \label{H_{X}^{alpha,tau}semi-norm}
 	\|v\|_{H_{X,\bsigma_2}^{\alpha,\tau}}
 	:= \
 	\brac{\sum_{k \in \NN}	
 		\brac{2^{\alpha k}k^{-\tau}  \norm{v - P_k^X v} {H_{X,\bsigma_2}}}^2}^{1/2} < \infty.
 \end{equation} 
\black{ Note that in the definition \eqref{H_{X}^{alpha,tau}semi-norm}, the subspaces $(V_k)_{k \in \NN}$ is not required to be nested and the projection operators $(P_k^X)_{k \in \NN}$ not to be required orthogonal.}
 
Later on we will see that under certain assumptions the parametric solution $u$ to the parametric elliptic PDE \eqref{parametricPDE} with log-normal inputs \eqref{lognormal} belongs to a certain space   $H_{X}^{\alpha}$.  However, such a space has not Hilbertian structure and hence  is not appropriate for constructing least squares sampling algorithms. Due to the continuous embedding 
$H_{X}^{\alpha} \hookrightarrow H_X^{\alpha,\tau}$ for any $\tau > \black{1/2}$, the parametric solution $u$ can be treated as an element of  the  Hilbert space $H_X^{\alpha,\tau}$ for which we are able to construct least squares sampling algorithms.
 
 	Let $A^X$ be a  linear operator in $L_2(U,X;\mu)$ defined for $v\in L_2(U,X;\mu)$ by 	
 \begin{equation} \label{A^X}
 	v \mapsto \sum_{k \in \NN}  
 	\brac{\sum_{s \in \NN} a_{k,s}	v_s}\varphi_k,
 \end{equation}
 where $(a_{k,s})_{(k,s) \in \NN^2}$	 is an infinite dimensional matrix. Then $A^X$ uniquely defines the linear operator $A^\CC$ in $L_2(U,\CC;\mu)$  given for $f\in L_2(U,\CC;\mu)$ by 	
 \begin{equation} \label{A^CC}
 	f \mapsto \sum_{k \in \NN}  
 	\brac{\sum_{s \in \NN} a_{k,s}	\black{f_s}}\varphi_k.
 \end{equation}
 \black{
 Conversely, a linear operator $A^\CC$ in $L_2(U,\CC;\mu)$  of the form \eqref{A^CC} uniquely defines the $A^X$  linear operator of the form \eqref{A^X} in $L_2(U,X;\mu)$, i.e., this is an one-onto-one correspondence between the linear operators in $L_2(U,X;\mu)$ and $L_2(U,\CC;\mu)$. 
 	If $A^X$ is  linear and bounded, then there holds the equality \cite[Lemma 2.1]{BD2024}
 	\begin{equation} \label{A^X=A^CC}
 		\big\|A^X\big\|_{H_{X,\bsigma}\to L_2(U,X;\mu)}
 		=
 		\big\|A^{\CC}\big\|_{H_{\CC,\bsigma} \to L_2(U,\CC;\mu)}.
 	\end{equation}
 
These correspondence and norm equality between bounded linear operators $A^X$ and $A^\CC$ allows us to define the spaces $H_\CC^{\alpha}$ and $H_{\CC}^{\alpha,\tau}$ generated from given spaces $H_X^{\alpha}$ and $H_{X}^{\alpha,\tau}$, by replacing $P_k^X$ with $P_k^\CC$ and $H_{X,\bsigma_i}$ with $H_{\CC,\bsigma_i}$, $i=1,2$,  in their  definitions 
\eqref{H_{X}^{alpha}}--\eqref{H_{X}^{alpha}semi-norm} and  \eqref{H_{X}^{alpha,tau}}--\eqref{H_{X}^{alpha,tau}semi-norm}, respectively. 

For convenience, in what follows, we will use the letter $X$ to denote a separable Hilbert space or the complex plane $\CC$, and  use the abbreviations $P_k:=P_k^X$ and  $P_k:=P_k^\CC$, provided this does not lead to any misunderstanding. 
}

 We will need the following lemma.
 
 \begin{lemma}\label{lemma:norms-equality}
 	We have 
 	\begin{equation*} 
 		\|v- P_kv\|_{H_{X,\bsigma_2}}
 		\ = \
 		\brac{	\sum_{s\in\NN} 
 			\brac{\sigma_{2;s} \|v_s - P_kv_s\|_{X}}^2}^{1/2}, \ \ 	v \in H_{X,\bsigma_2}.
 	\end{equation*}
 \end{lemma}
 
 \begin{proof} By the definition,
 	\begin{equation} \nonumber
 		\|v- P_kv\|_{H_{X,\bsigma_2}}
 		:= \
 		\brac{\sum_{s \in \NN} \brac{\black{\sigma_{2;s}} \|(v- P_kv)_s\|_X}^2}^{1/2} < \infty.		
 	\end{equation}
 	Due to the equality 	$(v - P_kv)_s= v_s - (P_kv)_s$, to prove the lemma it is sufficient to show $(P_kv)_s= P_kv_s$.
 Since $\bsigma_2^{-1}\in \ell_2(\NN)$,  \black{the series} \eqref{series} converges unconditionally  in  $L_2(U,X;\mu)$  to $v$ for $v \in H_{X,\bsigma_2}$.
 	Indeed, by the H\"older inequality we have for $v \in H_{X,\bsigma_2}$, 
 	\begin{equation} \nonumber
 		\begin{split}
 			\sum_{s\in \NN} \|v_s  \varphi_s\|_{L_2(U,X;\mu)}
 			\ &= \
 			\sum_{s\in\NN} \|v_s\|_{X} \|\varphi_s\|_{L_2(U, \CC;\mu)}
 			\ = \
 			\sum_{s\in\NN} \|v_s\|_{X}
 			\\
 			&\leq 
 			\bigg( \sum_{s\in \NN} (\black{\sigma_{2;s}}\|v_s\|_{X})^2\bigg)^{1/2} 
 			\bigg(\sum_{s\in \NN} \black{\sigma_{2;s}}^{-2}\bigg)^{1/2} \ < \ \infty.		
 		\end{split}
 	\end{equation}
 	This yields that the series \eqref{series} absolutely and hence unconditionally converges in $L_2(U, \CC;\mu)$ to $v$, since in a Banach space the absolute convergence implies the unconditional convergence. From this unconditional convergence  and the uniform boundedness of the linear projectors $(P_k)_{k \in \NN}$ we derive
 	\begin{equation} \nonumber
 		P_k v
 		\ = \
 		\sum_{s \in \NN} P_k v_s \varphi_s,
 	\end{equation}
 	and therefore, 	$(P_k v)_s= P_k v_s$.
 	\hfill
 \end{proof}
 
	\subsection{Extended least squares sampling algorithms in Bochner spaces}
	\label{Extended least squares sampling algorithms in Bochner spaces}

We study the approximate recovery of functions in the space 
	$H_{X,\bsigma}$ or $H_{X}^{\alpha,\tau}$
	from a finite set of their samples. 
	Throughout this paper, in the context of sampling recovery,  the inclusion $v \in H_{X,\bsigma}$
	or $v \in H_{X}^{\alpha,\tau}$ means that $v$ is a representative of an element from $L_2(U,X;\mu)$ that the pointwise evaluations $v(\by)$ for every
	$\by \in U_0$  are well-defined with $U_0$ fixed.
	
	Let us  formulate the problem of  sampling recovery for 
	$v \in H_{X,\bsigma}$ or $H_{X}^{\alpha,\tau}$ as follows:
	Given sample points $\by_1,\ldots,\by_n \in U_0$ and functions $h_1,\ldots,h_n \in L_2(U,\CC;\mu)$, we consider the approximate recovery of $v$ from its values $v(\by_1),\ldots, v(\by_n)$ by the linear sampling algorithm $S_n^X$ on $U$ defined as
	\begin{equation} \label{S_k}
		(S_n^X v)(\by): = \sum_{i=1}^n v(\by_i) h_i(\by).
	\end{equation}
	For convenience, we assume that some of the sample points $\by_i$ may coincide. 
	
	Let $F$ be a subset of $H_{X,\bsigma}$ or $H_{X}^{\alpha,\tau}$. Let $\Ss_n^X$ be the family of all linear sampling algorithms 
	$S_k^X$ in $L_2(U,X;\mu)$ of the form \eqref{S_k} with $k \le n$.
	To study the optimality of linear sampling algorithms from 
	$\Ss_n^X$ for the set $F$ 
	and their convergence rates we use the (linear) sampling $n$-width
	\begin{equation*} \label{rho_n}
		\varrho_n(F, L_2(U,X;\mu)) :=\inf_{S_n^X \in \Ss_n^X} \ \sup_{v\in F} 
		\|v - S_n^X v\|_{L_2(U,X;\mu)}.
	\end{equation*}

Denote by $B_{X,\bsigma}$, $B_X^{\alpha}$ and $B_X^{\alpha,\tau}$ the unit balls in the spaces $H_{X,\bsigma}$, $H_X^{\alpha}$ and $H_X^{\alpha,\tau}$, respectively.

\black{From \eqref{A^X=A^CC} it follows that \cite[Theorem 2.1]{BD2024}}
\begin{lemma}\label{lemma:sampling-equality}
	Given arbitrary sample points $\by_1,\ldots,\by_n \in \black{U_0}$ and functions $h_1,\ldots,h_n \in L_2(\RRi,\CC;\mu)$,	for the sampling algorithm $S_n^X$  in $L_2(U,X;\mu)$ defined by \eqref{S_k}, we have
	\begin{equation*} \label{sampling-equality}
		\sup_{v \in B_{X,\bsigma}} \norm{v- S_n^X v}{L_2(U,X;\mu)}
		= \sup_{f \in B_{\CC,\bsigma}} \norm{f - S_n^\CC f}{L_2(U,\CC;\mu)}.
	\end{equation*}		
\end{lemma}


For $k \in \NN$ and $v \in H_X^{\alpha,\tau}$,  we define
\begin{equation} \label{delta_k}
	\delta_k v
	:= \
	P_{k} v  - P_{k-1} v, \ \black{k \ge 2}, \quad \delta_1 v = P_1 v.
\end{equation}
We then can represent every $v \in H_X^{\alpha,\tau}$  by the series
\begin{equation*} 
	v
	\ = \
	\sum_{k \in \NN}\delta_k v
\end{equation*}
converging in $L_2(U,X;\mu)$ absolutely and hence, unconditionally, and satisfying 
\begin{equation} \label{norm-eq}
\black{	\brac{\sum_{k \in \NN}	
		\brac{2^{\alpha k}k^{-\tau}  \norm{\delta_k v}  {H_{X,\bsigma_2}}}^2}^{1/2}
	\ \le \
	2\|v\|_{H_{X,\bsigma_2}^{\alpha,\tau}}
	\ < \infty. }
\end{equation}

For $m \in \NN$, we introduce the sets
\begin{equation*} \label{Lambda_sigma(xi)}
	\delta_mB_X^{\alpha,\tau}
	:= \ 
	\ \big\{\delta_m v:\,  v \in B_X^{\alpha,\tau}\big\}, \ \ \
	\delta_mB_\CC^{\alpha,\tau}
	:= \ 
	\ \big\{\delta_m f:\,  f \in B_\CC^{\alpha,\tau}\big\}.
\end{equation*}	

\begin{lemma} \label{lemma:OperatorNormEquality}
	\black{Let $A^X$ be a bounded linear operator in $L_2(U,X;\mu)$.}
	Assume that there is a constant $K$ such that for every $v \in H_X^{\alpha,\tau}$ and every $m \in \NN$, $\norm{P_m v}{H_X^{\alpha,\tau}} \le K\norm{v}{H_X^{\alpha,\tau}}$.  
Then for every $m \in \NN$ there holds the equality
	\begin{equation} \label{OperatorNormEquality}
	\sup_{0 \not=\delta_m v\in \delta_mB_X^{\alpha,\tau}}
	\frac{\big\|A^X\delta_m v\big\|_{L_2(U,X;\mu)}}{\norm{\delta_m v}{H_X^{\alpha,\tau}}}
		\ = \
		\sup_{0 \not=\delta_m f\in \delta_mB_\CC^{\alpha,\tau}}	
	\frac{\big\|\black{A^\CC}\delta_m f\big\|_{L_2(U,\CC;\mu)}}{\norm{\delta_m f}{H_\CC^{\alpha,\tau}}}.	
	\end{equation}
\end{lemma}

\begin{proof} From the  assumptions we can see that for every $v \in H_X^{\alpha,\tau}$ and every $m \in \NN$, $\norm{\delta_m v}{H_X^{\alpha,\tau}} \le 2 K\norm{v}{H_X^{\alpha,\tau}}$. 
	Denote by $N_X$ and $N_\CC$ the left-hand side and right-hand side of \eqref{OperatorNormEquality}, respectively. 
	For $\delta_m f \in \delta_m B_\CC^{\alpha,\tau}$ represented as
	\begin{equation} \nonumber
		\delta_m f
		= \
		\sum_{s \in \NN} 	(\delta_m f)_s \varphi_s,
	\end{equation}
	we have 
	\begin{equation} \nonumber
		\|A^{\CC} \delta_m f\|_{L_2(U,\CC;\mu)}^2
		\le
		N_\CC^2 	
		\brac{\|\delta_m f\|_{H_{\CC,\bsigma_1}}^2 + \|\delta_m f\|_{H_{\CC,\bsigma_2}^{\alpha,\tau}}^2}.
	\end{equation}
	The last inequality is equivalent to inequality
	\begin{equation*} 
		\sum_{i \in \NN}  
		\left|\sum_{s \in \NN} a_{i,s} (\delta_m f)_s\right|^2
		\le
	N_\CC^2 	
		\brac{\sum_{s\in\NN}
			\brac{ \sigma_{1;s} |(\delta_m f)_s|}^2  
			+	\sum_{k \in \NN}	\sum_{s\in\NN} 
			\brac{2^{\alpha k}k^{-\tau} \sigma_{2;s} |(\delta_m f - P_k (\delta_m f))_s|}^2}. 
	\end{equation*}
	For $\delta_m v \in \delta_mH_X^{\alpha,\tau}$, we have 
	\begin{equation} \nonumber
		\delta_m v
		= \
		\sum_{s\in \NN} 
		(\delta_m v)_s \varphi_s
\end{equation}
with			
	\begin{equation} \nonumber	
		\sum_{s\in\NN}
		\brac{ \sigma_{1;s} \|(\delta_m v)_s\|_X}^2  
		+	\sum_{k \in \NN}	\sum_{s\in\NN} 
		\brac{2^{\alpha k}k^{-\tau} \sigma_{2;s} \|(\delta_m v - P_k (\delta_m v))_s\|_X}^2 
		< \infty,
	\end{equation}	
	and 
	\begin{equation} \label{A^Xdelta_m v=}
		\|A^X\delta_m v\|_{L_2(U,X;\mu)}^2
		= \
		\sum_{i \in \NN} 	\left\|\sum_{s \in \NN} a_{i,s} (\delta_m v)_s \right\|_X^2.
	\end{equation}
	Let $\brac{\psi_j}_{j \in \NN}$ be an orthonormal basis of  $X$.  Then $\brac{\varphi_k \psi_j}_{k,j \in \NN}$ is an orthonormal basis of  $L_2(U,X;\mu)$. We have
	\begin{equation} \nonumber
		\delta_m v
		= \
		\sum_{j \in \NN} 	\sum_{s \in \NN} 
		(\delta_m v)_{s}^j\psi_j \varphi_s 
		= 
		\sum_{j \in \NN} 	
		(\delta_m v^j)\psi_j.
	\end{equation}	
	Then,
	\begin{equation} \nonumber
		A^X\delta_m v
		= \
		\sum_{s \in \NN} \sum_{i \in \NN} \sum_{j \in \NN} a_{i,s}(\delta_m v)^j_s \psi_j  \varphi_m.
	\end{equation}
	By applying  \eqref{A^Xdelta_m v=}  to $\delta_m f = \delta_m v^j$, we obtain for every $k \in \NN$,
	\begin{equation} \nonumber
		\begin{aligned}
			&	\|A^X\delta_m v\|_{L_2(U,X;\mu)}^2
			\\ &	= \
			\sum_{j \in \NN}\sum_{i \in \NN}
			\left|\sum_{s \in \NN}  a_{i,s} 	(\delta_m v)_{s}^j  \right|^2
			= \sum_{j \in \NN}\|A^\CC (\delta_m v)^j\|_{L_2(U,\CC;\mu)}^2
			\\&
			\le \
			\black{N_\CC^2}
			\brac{\sum_{j \in \NN} 
				\sum_{s\in\NN}
				\brac{ \sigma_{1;s} |((\delta_m v)^j)_s|}^2  
				+	\sum_{j \in \NN} \sum_{k \in \NN}	\sum_{s\in\NN} 
				\brac{2^{\alpha k}k^{-\tau} \sigma_{2;s} |((\delta_m v)^j - P_k ((\delta_m v)^j))_s|}^2}
			\\&
			=:  	\black{N_\CC^2}
			(B_1^2 + B_2^2).
		\end{aligned}	
	\end{equation}	
	For the term $B_2^2$, we have
	\begin{equation} \nonumber
		\begin{aligned}
			B_2^2
			&= \		
			\sum_{k \in \NN}\sum_{s \in \NN}\brac{2^{\alpha k}  k^{-\tau}\sigma_{2;s}}^2 
			\sum_{j\in\NN} |((\delta_m v)^j - P_k ((\delta_m v))^j)_s|^2	
			\\&
			=
			\sum_{k \in \NN}	\sum_{s\in\NN} 
			\brac{2^{\alpha k} k^{-\tau}\sigma_{2;s} |((\delta_m v)^j - P_k ((\delta_m v))^j)_s|}^2
			=
			\|\delta_m v\|_{H_{X,\bsigma_2}^{\alpha,\tau}}^2.
		\end{aligned}	
	\end{equation}	
	By the same way we can show
	\begin{equation} \nonumber
		B_1
		\ = \		
		\|\delta_m v\|_{H_{X,\bsigma_1}}.
	\end{equation}	
	Summing up 	we prove  the inequality	
	\begin{equation} \nonumber
	N_X
		\le \  N_{\CC}.
	\end{equation}
	In order to prove the inverse inequality, let 
	$\brac{f^{\black{(n)}}}_{n \in \NN} \subset H_\CC^{\alpha,\tau}$ be  a sequence  such that 	 
	$\|\black{\delta_m f^{(n)}}\|_{H_\CC^{\alpha,\tau}} = 1$, and 
	$$
	\lim_{n \to \infty} \big\|A^{\CC}\black{\delta_m f^{(n)}}\big\|_{L_2(U,\CC;\mu)} = N_{\CC}.
	$$
	\black{Define $v^{(n)}:= f^{(n)}\psi_1$. Then 
	$\delta_m v^{(n)}= \delta_m f^{(n)}\psi_1$, \  
	 $\|\delta_m v^{(n)}\|_{H_X^{\alpha,\tau}}= 1$} and
	\begin{equation} \nonumber
		\begin{aligned}
		\big\|A^X \black{\delta_m v^{(n)}}\big\|_{L_2(U,X;\mu)}^2
			&	= \
			\sum_{k \in \NN} 	
			\left\|\sum_{s \in \NN} a_{k,s}	\langle \black{\delta_m f^{(n)}},\varphi_s \rangle_{L_2(U, \CC; \mu)} \psi_1 \right\|_X^2
			\\&
			= \
			\sum_{k \in \NN} 	
			\left|\sum_{s \in \NN} a_{k,s}	\langle \black{\delta_m f^{(n)}},\varphi_s \rangle_{L_2(U, \CC; \mu)} \right|^2
			\\&
			= \
			\big\|A^{\CC} \black{\delta_m f^{(n)}}\big\|_{L_2(U,\CC;\mu)}^2	 \ \to 	\	\black{N_\CC^2} \ \ \text{as} \ \ n \to \infty.
		\end{aligned}	
	\end{equation}	
	This proves the inequality	
	\begin{equation} \nonumber
		N_X
		\ge \ 
	N_\CC.
	\end{equation}
	\hfill
\end{proof}

From this lemma we can extend the  result of Lemma \ref{lemma:sampling-equality} to the sets $\delta_mB_X^{\alpha,\tau}$ and $\delta_mB_\CC^{\alpha,\tau}$.

\begin{corollary}\label{corollary:OperatorNormEquality}
	Assume that there is a constant $K$ such that for every $v \in H_X^{\alpha,\tau}$ and every $m \in \NN$, $\norm{P_m v}{H_X^{\alpha,\tau}} \le K\norm{v}{H_X^{\alpha,\tau}}$. 
	Given arbitrary sample points $\by_1,\ldots,\by_{\black{n}} \in U$ and functions 
	$h_1,\ldots,h_{\black{n}} \in L_2(U,\CC;\mu)$,	for the sampling algorithm $S_n^X$  in $L_2(U,X;\mu)$ defined by \eqref{S_k}, we have for every $m \in \NN$,
	\begin{equation} \nonumber
\sup_{0 \not=\delta_m v\in \delta_mB_X^{\alpha,\tau}}			
		\frac{\norm{\delta_m v - S_n^X (\delta_m v)}{L_2(U,X;\mu)}}{\norm{\delta_m v}{H_X^{\alpha,\tau}}}
		\ = \
			\sup_{0 \not=\delta_m f\in \delta_mB_\CC^{\alpha,\tau}}	
		\frac{\norm{\delta_m f - S_n^\CC (\delta_m f)}{L_2(U,\CC;\mu)}}{\norm{\delta_m f}{H_\CC^{\alpha,\tau}}}.	
	\end{equation}
\end{corollary}

\begin{proof}
This corollary is  Lemma \ref{lemma:OperatorNormEquality} for $A^X = I^X - S_n^X$, where
	$I^X$  denotes the identity operator in $L_2(U,X;\mu)$. 
	\hfill
\end{proof}

Let $n \in \NN$ and $E$ be a normed space and $F$ a central symmetric compact set in $E$. 
Then the Kolmogorov $n$-width of $F$ is defined by
\begin{equation*}
	d_n(F,E):= \ \inf_{L_{n}}\sup_{f\in F}\inf_{g\in L_n}\|f-g\|_E,
\end{equation*}
where the left-most infimum is taken over all subspaces $L_{n}$ of dimension at most $n$ in $E$. 
The Kolmogorov $n$-width $d_n(F,E)$ is a characterization of the best approximation of elements in $F$ by elements in linear subspaces of dimension at most $n$.

We recall a concept  of  weighted least squares sampling algorithm in the space $L_2(U,\CC;\mu)$ (cf. also \cite{BD2024}). Recall that  $\brac{\varphi_s}_{s \in \NN}$ is the fixed orthonormal basis of  $L_2(U,\CC;\mu)$ considered in Subsection~\ref{Extended least squares sampling algorithms in Bochner spaces}.
For $n,m\in\NN$ with $n\ge m$, let 
$\by_1, \dots, \by_n\in U$ be points, $\omega_1, \dots, \omega_n\ge 0$ be weights, and 
$\Phi_m = \operatorname{span}\{\varphi_s\}_{s=1}^{m}$ the subspace spanned by the functions $\varphi_s$,   $s=1,...,m$.
The weighted least squares sampling algorithm 
$$
\black{S_n^{\CC} f := S_n^{\CC}(\by_1, \dots, \by_n, \omega_1, \dots, \omega_n, \Phi_m) f}
$$ 
of a function $f\colon U\to\CC$, is given by
\begin{equation} \label{least-squares-sampling1}
	S_n^{\CC} f
	= \operatorname{arg\,min}_{g\in \Phi_m} \sum_{i=1}^n \omega_i |f(\by_i) - g(\by_i)|^2 .
\end{equation}
The least squares approximation can be computed using the \black{Moore-Penrose inverse, and it is the approximation} of smallest error for over-determined systems where no exact solution can be expected.
In particular, for $\boldsymbol L = [\varphi_s(\by_i)]_{i=1,\dots,\black{n}; \, s=1,\dots,m}$ and \black{$\boldsymbol W = \operatorname{diag}(\omega_1, \dots, \omega_n)$} we have
\black{
	\begin{equation} \label{least-squares-sampling2}
		S_{\black{n}}^\CC f
		= \sum_{s=1}^{m} \hat g_{s} \varphi_s
		\quad\text{with}\quad
		(\hat g_1, \dots, \hat g_{m})^\top
		= (\boldsymbol W^{1/2}\boldsymbol L)^{+} \boldsymbol W^{1/2} (f(\by_1), \dots, f(\by_{\black{n}}))^{\top},
	\end{equation} 
	where $(\boldsymbol W^{1/2}\boldsymbol L)^{+}$ denotes the pseudo-inverse of $\boldsymbol W^{1/2}\boldsymbol L$.
	In the presented theorems the setting is such that the matrix $\boldsymbol W^{1/2}\boldsymbol L$ has full rank and we have
	$(\boldsymbol W^{1/2}\boldsymbol L)^{+} = (\boldsymbol L^\ast\boldsymbol W\boldsymbol L)^{-1}\boldsymbol L^\ast\boldsymbol W^{1/2}$.
}

Notice that $S_n^\CC$ is a linear sampling algorithm of the form
\begin{equation}\label{least-squares3}
	S_n^\CC f
	= \sum_{i=1}^n f(\by_i) h_i(\by)
\end{equation}
for some $h_1,..., h_n$ in $L_2(U,\CC;\mu)$.
The extended least squares algorithm to the Bochner space $L_2(U,X;\mu)$ can be defined 
by replacing $f\in L_2(U,\CC;\mu)$ with $v\in L_2(U,X;\mu)$:
\begin{equation}  \label{least-squares-sampling-extension}
\black{	S_n^X v
	:=
	S_{\black{n}}^X(\by_1, \dots, \by_{\black{n}}, \omega_1, \dots, \omega_{\black{n}}, \Phi_m) v
	:= \sum_{i=1}^n v(\by_i) h_i(\by).}
\end{equation}


The following result \cite[Theorem 3]{DKU2023} (see also its proof) gives a error bound of the approximation by 
\black{$S_n^\CC$} via Kolmogorov $n$-widths.

\begin{lemma}\label{lemma:sampling inequality}
	Let $0 < p <2$ and $d_s := d_s(F,L_2(U,\CC;\mu))$. Assume that $F \subset L_2(U,\CC;\mu)$ is such that there is a metric on $F$ satisfying the condition that $F$ is continuously embedded into $L_2(U,\CC;\mu)$, and the function evaluation $f \mapsto f(\by)$ is continuous on $F$ for every $\by \in F$. 
	Then there is a constant $c_p  \in \NN$ such that 
	for any $n\in \NN$, there exist points 
	$\by_1, \dots, \by_n\in U_0$ and weights  $\omega_1, \dots, \omega_n$ such that  		
			\begin{equation} \nonumber
				\varrho_n(F, L_2(U,\CC;\mu)) 
				\ \le \
			\sup_{f\in F} \norm{\black{f}- S_n^\CC f}{L_2(U,\CC;\mu)}
			\ \le \brac{\frac{c_p}{k} \sum_{s \ge k/8} d_s^p}^{1/p},
		\end{equation}	
		where 
		\begin{equation}\label{tilde{S}_n^CC:=}
			S_n^\CC
			:= 	
			S_n^\CC(\by_1, \dots, \by_n, \omega_1, \dots, \omega_n, \Phi_k)
			\quad\text{and}\quad
			n\ge k \ge \frac{n}{43200}\,. 
		\end{equation}			
\end{lemma}

\black{
For $m\in\NN$ let the probability  measure $\nu = \nu(m)$ be defined by
\begin{equation}\label{nu}
	\mathrm d\nu(\by)
	:= \varrho(\by)\mathrm d\mu(\by)
	:= \frac{1}{2}\left( \frac{1}{m}\sum_{s=1}^{m}|\varphi_s(\by)|^2
	+ \frac{\sum_{s=m+1}^{\infty}|\sigma_{s}^{-1}\varphi_s(\by)|^2}{\sum_{s=m+1}^{\infty}\sigma_{s}^{-2}} \right)\mathrm d\mu(\by) .
\end{equation}
}
 The sampling algorithms \black{$S_n^\CC$}  in Lemma \ref{lemma:sampling inequality} is a weighted least squares algorithm using samples from a set of $n$ points that is subsampled from a set of $ n \log n$ i.i.d. random points with respect to the probability measure $\nu$. Depending on the function class $F$, the algorithm using the full set of random points may be constructive but the subsampling is  not constructive. We will explain and discuss it and other least squares sampling algorithms  in detail in Section \ref{Extensions}.

\begin{lemma}\label{lemma:sampling inequalityX}
Let $0 < p <2$ and $d_s := d_s(B_{\CC,\bsigma},L_2(U,\CC;\mu))$. 
	Then there is a constant $c_p  \in \NN$ such that 
for any $n\in \NN$, there exist points 
$\by_1, \dots, \by_n\in U_0$ and weights  $\omega_1, \dots, \omega_n$ such that  		
\begin{equation} \nonumber
	\varrho_n(B_{X,\bsigma}, L_2(U,X;\mu)) 
	\ \le \
	\sup_{v \in B_{X,\bsigma}} \norm{v - S_n^X v}{L_2(U,X;\mu)}
	\ \le \brac{\frac{c_p}{k} \sum_{s \ge k/8} d_s^p}^{1/p},
\end{equation}	
where 
\begin{equation}\label{tilde{S}_n^X:=}
	S_n^X
	:= 	
	S_n^X(\by_1, \dots, \by_n, \omega_1, \dots, \omega_n, \Phi_k)
	\quad\text{and}\quad
	n\ge k \ge \frac{n}{43200}\,. 
\end{equation}			
\end{lemma}

\begin{proof}
	We first consider
	the particular case when $X = \CC$. In this case $B_{\CC,\sigma}$ is a subset of the unit ball  of the reproducing kernel Hilbert space $H_{\CC,\bsigma}$ equipped with a metric as the norm of  this space, which is continuously embedded into $L_2(U,X;\mu)$  and for which  the function evaluation $f \mapsto f(\by)$ is continuous on $B_{\CC,\sigma}$ for every $\by \in B_{\CC,\sigma}$. This means that the assumption of Lemma~\ref{lemma:sampling inequality}  is satisfied for $F= B_{\CC,\sigma}$. The lemma has been proven for the particular case when $X = \CC$.
	Hence, by 
	using  Lemma \ref{lemma:sampling-equality} we prove the lemma in the general case.
	\hfill
\end{proof}

In a similar way from Lemma \ref{lemma:sampling inequality} and Corollary \ref{corollary:OperatorNormEquality} we derive 

\begin{lemma}\label{lemma:sampling inequality-delta_m}
	Let $0 < p <2$, $m \in \NN$ and $d_s:= d_s(\delta_mB_\CC^{\alpha,\tau},L_2(U,\CC;\mu))$. 	 Assume that there is a constant $K$ such that for every $v \in H_X^{\alpha,\tau}$ and every $m \in \NN$, 
	$\norm{P_m v}{H_X^{\alpha,\tau}} \le K\norm{\black{v}}{H_X^{\alpha,\tau}}$. 
	Then there is a constant $c_p  \in \NN$ such that 
for any $n\in \NN$, there exist points 
$\by_1, \dots, \by_n\in U_0$ and weights  $\omega_1, \dots, \omega_n$ such that 
	\begin{equation} \nonumber
			\varrho_n(\delta_mB_X^{\alpha,\tau}, L_2(U,X;\mu)) 
		\ \le \
		\sup_{\delta_m v \in \delta_mB_X^{\alpha,\tau}} 
		\norm{\delta_m v - S_n^X (\delta_m v)}{L_2(U,X;\mu)}
		\ \le \brac{\frac{c_p}{k} \sum_{s \ge k/8} d_s^p}^{1/p},
	\end{equation}		
	where $S_n^X$ \black{and $k$ are} defined as in \eqref{tilde{S}_n^X:=}.
\end{lemma}


Let $\Lambda \subset \NN$ be a finite set. \black{Let 
$$
\Vv(\Lambda) = \operatorname{span}\{\varphi_s:\, s \in \Lambda\}
$$ 
be the subspace of $L_2(U,X;\mu)$, spanned by the functions $\varphi_s$,   $s \in \Lambda$. For
$v \in L_2(U,X;\mu)$,
we define the truncation $S_{\Lambda}v$  belonging to $\Vv(\Lambda)$, of the  GPC expansion of $v$ by}
\begin{equation*} 
	S_{\Lambda}v := \sum_{s\in\Lambda} \black{v_s} \varphi_s \in \Vv(\Lambda).
\end{equation*}
The following lemma has been proven in \black{\cite[Lemma 2.3]{BD2024}}.

\begin{lemma} \label{lemma:d_n><}
	Let $0 < q\le 2$ and $\bsigma = (\sigma_{s})_{s \in \NN}$  a given sequence of numbers strictly larger than $1$, $\bsigma^{-1} \in \ell_q(\NN)$. For $\xi>0$ and $M > 0$, we introduce the set	
	\begin{equation*} \label{Lambda_sigma(xi)}
		\Lambda(\xi)
		:= \ 
		\ \big\{s \in \NN: \, \sigma_s \le \xi^{1/q} \big\}.
	\end{equation*}
	We have
	\begin{equation} \label {sup_{B_{X,bsigma}^q}}
		\sup_{v \in B_{X,\bsigma}}	\|v- S_{\Lambda(\xi)}v\|_{L_2(U,X;\mu)}		
			\ \le \ \xi^{-1/q}.
	\end{equation} 
Moreover, if in addition $\norm{\bsigma^{-1}}{\ell_q(\NN)} \le 1$, then for all $n \ge 2$ there exists a number $\xi_n$  such that 
$
\dim(\Vv(\Lambda(\xi_n))  \le  n,
$
 and	
		\begin{equation} \label{d_n><}
		d_n(B_{\CC,\bsigma},L_2(U,\CC;\mu))
				\ \le \ 
		\sup_{f \in B_{\CC,\bsigma}}	
		\|f- S_{\Lambda(\xi_n)}f\|_{L_2(U,\CC;\mu)}		
		\ \le\black{(n+1)}^{-1/q} \ \ \forall n \in \NN.
	\end{equation}
\end{lemma}

\begin{lemma} \label{lemma:d_n(delta_kB_X)}
	Let $0 < q_1 \le q_2 <2$ and $\alpha > 1/q_2$. 
	Let the numbers $\beta$ and $\delta$ be defined as 
	\begin{equation} 	\label{delta, beta}
		\beta := \frac 1 {q_1}\frac{\alpha}{\alpha + \delta}, \quad 
		\delta := \frac 1 {q_1} - \frac 1 {q_2}.
	\end{equation}	
	 Assume that there is a constant $K$ such that for every $v \in H_X^{\alpha,\tau}$ and every $m \in \NN$, $\norm{P_m v}{H_X^{\alpha,\tau}} \le K\norm{v}{H_X^{\alpha,\tau}}$.  For $\xi>1$ and $m \in \NN$, we introduce the set
	 \begin{equation*} 
	 	\Lambda_m(\xi)
	 	:= \ 	 	
	 		\big\{\black{s \in \NN:} \ \sigma_{2;s} \leq \xi^{1/q_1}2^{-\alpha m}m^\tau, \ \sigma_{1;s}\le \xi^{1/q_1} \big\}. 
	 \end{equation*}
	We have
	\begin{equation} \label{sup_{B_{X}}}
		\|\delta_m v - S_{\Lambda(\xi)}(\delta_m v)\|_{L_2(U,X;\mu)}		
		\ \le \ 2K\xi^{-2/q_1}\|\black{v}\|_{H_X^{\alpha,\tau}}, \ \ \ v \in H_X^{\alpha,\tau}.
	\end{equation} 
	Moreover, for all $n \ge 2$ \black{and $m \in \NN$} there exists a number $\xi_n$  such that 
	\black{$|\Lambda_m(\xi_n)|  \le  n$},
	and	
	\begin{equation} \label{d_n(delta_mB)}
\black{\sup_{v \in B_X^{\alpha,\tau}}
		\|\delta_m v -  S_{\Lambda_m(\xi_n)}(\delta_m v)\|_{L_2(U,X;\mu)}
		\ \le \  C\, (2^m n)^{-\beta} (m + \log n)^{\beta \tau /\alpha},}
	\end{equation}
\black{where the positive constant $C$ depends on $\alpha, q_1,q_2$ and $K$ only.}
\end{lemma}	

\begin{proof}	 From the  assumptions we can see that for every $v \in H_X^{\alpha,\tau}$ and every $m  \in \NN_0$, 
	$$\norm{\delta_m  v}{H_X^{\alpha,\tau}} \le 2 K\norm{v}{H_X^{\alpha,\tau}}.
	$$
	 For a function \black{$v\in B_X^{\alpha,\tau}$},
	 putting
	\[
	(\delta_m  v)_\xi
	:=
	\sum_{\black{s \in \NN:} \, \sigma_{1;s}^{q_1} \le \xi} (\delta_m  v)_s \varphi_s,
	\]
	we get
	\begin{equation} \nonumber
		\|(\delta_m  v)- S_{\Lambda_m (\xi)} (\delta_m  v)\|_{L_2(U,X;\mu)}
		\ \le \
		\|(\delta_m  v)- (\delta_m  v)_\xi\|_{L_2(U,X;\mu)} + \|(\delta_m  v)_\xi - S_{\Lambda_m (\xi)} (\delta_m  v)\|_{L_2(U,X;\mu)}.
	\end{equation}
	For the norm $\|(\delta_m  v)- (\delta_m  v)_\xi\|_{L_2(U,X;\mu)}$ we have that
	\begin{equation*} 
		\begin{split}
			\|(\delta_m  v)- (\delta_m  v)_\xi\|_{L_2(U,X;\mu)} ^2
			\ &= \
			\sum_{\black{s \in \NN:} \, \sigma_{1;s}> \xi^{1/q_1} } \|(\delta_m  v)_s\|_{X}^2 
			\ = \
			\sum_{\black{s \in \NN:} \, \sigma_ {1;s} > \xi^{1/q_1} } (\sigma_{1;s}\|(\delta_m  v)_s\|_{X})^2 \sigma_{1;s}^{-2}
			\\[1.5ex]
			\ &\le \
			\xi^{-2/q_1}
			\sum_{\black{s \in \NN:} \, \sigma_{1;s}> \xi^{1/q_1} } (\sigma_{1;s}\|(\delta_m  v)_s\|_{X})^2 
			\\[1.5ex]
			\ &\le \
			 (2K)^2\xi^{-2/q_1}\|\black{v}\|_{H_X^{\alpha,\tau}}^2.
		\end{split}
	\end{equation*}
	For the norm $\|(\delta_m  v)_\xi - S_{\Lambda_m (\xi)} (\delta_m  v)\|_{L_2(U,X;\mu)}$,  we obtain 
	\begin{equation} \nonumber
		\begin{split}
			\|(\delta_m  v)_\xi - S_{\Lambda_m (\xi)} (\delta_m  v)\|_{L_2(U,X;\mu)}^2
			\ &= \
			\sum_{\black{s \in \NN:} \, \sigma_{1;s} \le \xi^{1/q_1}, \ \sigma_{2;s} > \xi^{1/q_1}2^{-\alpha m }m ^\tau} \norm{ (\delta_m  v)_s}{X}^2
					\\&
			\ \le \
			\sum_{\black{s \in \NN:} \, \sigma_{1;s} \le \xi^{1/q_1}, \ \sigma_{2;s} > \xi^{1/q_1}2^{-\alpha m }m ^\tau} 
			\brac{ 2^{\alpha  m } m ^{-\tau}\xi^{-1/q_1}\sigma_{2;s} \norm{ (\delta_m  v)_s}{X}}^2
			\\[1.5ex]
			\ &\le \
			\xi^{-2/q_1} \sum_{s \in \NN} 
			\brac{ 2^{\alpha  m } m ^{-\tau}\sigma_{2;s} \norm{ (\delta_m  v)_s}{X}}^2
		\\[1.5ex]
		\ &\le \
		(2K)^2\xi^{-2/q_1}\|\black{v}\|_{H_X^{\alpha,\tau}}^2.
		\end{split}
	\end{equation}
	These estimates yield \eqref{sup_{B_{X}}}.
	
	For the dimension of the space $\Vv(\Lambda_m (\xi))$, 	with $q:= q_2 \alpha > 1$ and $1/q' + 1/q = 1$, 
 we have that
	\begin{equation} \nonumber
		\begin{split}
	\black{\ |\Lambda_m (\xi)|} 
		\ &= \ \sum_{s \in \Lambda_m (\xi)} 1
			\ = \ 
			2^{-m }\sum_{\black{s \in \NN:} \, \sigma_{1;s}\le \xi^{1/q_1}, 
				\ \sigma_{2;s} > \xi^{1/q_1}2^{-\alpha m }m ^\tau} 2^m  
			\\ \  &\le \
		2^{-m }\sum_{s \in\NN: \  \sigma_{1;s}\le \xi^{1/q_1}}
			\ \sum_{\black{k\in \NN_0:} \ \sigma_{2;s} > \xi^{1/q_1}2^{-\alpha k}k^\tau} 2^{k} 
			\\ \  &\le \
			2^{-m } \sum_{\black{s \in\NN:\ \sigma_{1; s}}\le \xi^{1/q_1}}   \xi^{1/(q_1\alpha)}
			\black{\brac{\log\xi}^{\tau/\alpha}} \sigma_{2;\black{s}}^{-1/\alpha} 
			\\ \ &\le \
			2^{-m } \xi^{1/(q_1 \alpha)}\brac{\log\xi}^{\tau/\alpha}
			\left(\sum_{\black{s \in \NN:} \, \sigma_{1;\black{s}}\le \xi^{1/q_1}}\sigma_{2;\black{s}}^{-q_2} \right)^{1/q}
			\left(\sum_{\black{s \in \NN:} \, \sigma_{1;\black{s}}\le \xi^{1/q_1}} 1 \right)^{1/q'}
			\\ \  &\le 
			2^{-m } \xi^{1/(q_1 \alpha)}\brac{\log\xi}^{\tau/\alpha}\left(\sum_{\black{s \in \NN}} \sigma_{2;\black{s}}^{-q_2} \right)^{1/q}
			\left(\sum_{\black{s \in \NN}} \xi \sigma_{1;\black{s}}^{-q_1}  \right)^{1/q'}
				\\ \  &\le 
			 2^{-m }\xi^{1 + \delta/\alpha}\brac{\log\xi}^{\tau/\alpha}.
		\end{split}
	\end{equation}
	
	For any $n \ge 2$, letting $\xi_n$ be a number satisfying the inequalities 
	\begin{equation} \label{[xi_n]3}
		 \xi^{1 + \delta/\alpha}\brac{\log\xi}^{\tau/\alpha}  2^{-m }
		\ \le \
		n
		\ < \ 
		2\xi^{1 + \delta/\alpha}\brac{\log\xi}^{\tau/\alpha}  2^{-m },
	\end{equation}
	we derive that  $ \dim \Vv(\Lambda_m (\xi_n)) \le  n$, \black{and moreover,
	\begin{equation*} 
		\log\xi
		\ \le \
	m  + \log n.
	\end{equation*}
	From the last inequality and the second inequality in \eqref{[xi_n]3} it follows that}
	\begin{equation} \nonumber
		\begin{split}
			\xi_n^{-1/q_1} 
			\ \le \
			C \, \brac{2^m  n (\log n + m )^{-\tau/\alpha}}^{-\frac{1}{q_1} \frac{\alpha}{\alpha + \delta}} 
			\ \le \  C\, (2^m  n)^{-\beta} (m  + \log n)^{\beta \tau /\alpha},
		\end{split}
	\end{equation}	
	\black{where $C$ is a positive constant depending on on $\alpha, q_1,q_2$ only.}
	This together with \eqref{sup_{B_{X}}} proves \eqref{d_n(delta_mB)}.	
		\hfill
\end{proof}

From Corollary \ref{corollary:OperatorNormEquality} and Lemma \ref{lemma:d_n(delta_kB_X)} we obtain 

\begin{corollary} \label{corollary:d_n(delta_kB_CC)}
	Under the assumption and notation of Lemma \ref{lemma:d_n(delta_kB_X)},
	we have
	\begin{equation*} 
		\|\delta_m f - S_{\Lambda(\xi)}(\delta_m f)\|_{L_2(U,\CC;\mu)}		
		\ \le \ 2K\xi^{-2/q_1}\|\black{f}\|_{H_\CC^{\alpha,\tau}}, \ \   \forall f \in H_\CC^{\alpha,\tau}.
	\end{equation*} 
	Moreover, for all $n \ge 2$ \black{and $m \in \NN$} there exists a number $\xi_n$  such that 
	$\dim(\Vv(\Lambda_m(\xi_n))  \le  n$,
	and	
	\begin{equation*} 
		\begin{split}
			d_n(\delta_m B_\CC^{\alpha,\tau},L_2(U,\CC;\mu))
			\ &\le \ 
			\sup_{\black{ f \in  B_\CC^{\alpha,\tau}}}
			\|\delta_m f -  S_{\Lambda_m(\xi_n)}(\delta_m f)\|_{L_2(U,\CC;\mu)}
			\\		
			\ &\le \  C\, (2^m n)^{-\beta} (m + \log n)^{\beta \tau /\alpha},
		\end{split}
	\end{equation*}
	\black{where the positive constant $C$ depends on $\alpha, \tau, q_1,q_2$ and $K$ only.}
\end{corollary}	

\subsection{Multi-level least squares sampling algorithms}

Let $0 < p < 2$ and $(\xi_{n})_{n \in \NN}$ be a sequence of increasing positive numbers whose values will be selected later, such that $\xi_{n} \to \infty$ as $n \to \infty$. For $n \in \NN$,  we consider the multi-level least squares  operator $\Ss_n^X$  defined by
\begin{equation}  \label{Ss_{n}^X}
	\Ss_{n}^X v
	\ = \ 
	\sum_{k=1}^{k_n}  S_{ \black{n_k}}^X (\delta_k v),
\end{equation}
where  
\begin{equation} \label{n_k,k_n}
n_k := \lfloor  \black{n 2^{-k}}\rfloor, \ \ \ k_n:= \lfloor \log \xi_n \rfloor,
\end{equation}
and
\begin{equation} \label{S_n_k}
	\black{S_{n_k}^X  v : = \sum_{i=1}^{n_k}  v(\by_{k,i}) h_{k,i}}
\end{equation}
\black{are the least squares sampling algorithms 
 as in Lemma~\ref{lemma:sampling inequality}.} 

Since \black{$\delta_k v(\by_{k,i})$ belongs to the subspace $V_k + V_{k-1}$ of dimension at most 
	\black{$ 2^{k-1} + 2^{k-2}$}, and the number of sample points in the algorithms  $S_{n_k}^X$ is $n_k$, the computational complexity} 
$\operatorname{Comp} \brac{\Ss_{n}^X }$ of the operator  $\Ss_{n}^X$ can be defined as 
\begin{equation}  \nonumber
\operatorname{Comp} \brac{\Ss_{n}^X}
	 := \ 
	\sum_{k=1}^{k_n} \brac{\black{2^{k-1} + 2^{k-2}}}n_k.
\end{equation}

\begin{theorem}\label{thm:sampling}
	Let $0 < q_1 \le q_2 <2$ and $p$ be a fixed number such that $1/q_2 < p <2$. Let the numbers $\beta$ and $\delta$ be defined as in \eqref{delta, beta}.		
\black{Let $\tau > 1/2$ be a fixed number.}
	 Let $(\xi_{n})_{n \in \NN}$ be a sequence of increasing positive numbers $\xi_n$ satisfying the condition
	\begin{equation} 	\label{xi_n}
			\begin{cases}
\xi_n \ \le \ n\ < \ 2\xi_n	\quad &{\rm if }  \ \alpha \le 1/q_2,\\
\xi_n^{q_1(\alpha + \delta)} \ \le \ n \ < \ 2\, \xi_n^{q_1(\alpha + \delta)} \quad  & {\rm if }  \ \alpha > 1/q_2. 
			\end{cases}
	\end{equation}
	 Assume that there is a constant \black{$K$} such that for every $v \in H_X^{\alpha,\tau}$ and every $m \in \NN$, 
	 $$
	 \norm{P_m v}{H_X^{\alpha,\tau}} \le K\norm{v}{H_X^{\alpha,\tau}}.
	 $$
	Then for all $n \ge 2$ there exist  $\by_{k,1},...,\by_{k,n_k} \in \black{U_0}$ and $h_{k,1},...,h_{k,n_k} \in L_2(U,\CC;\mu)$, $k=1,...,k_n$,  such that for the operator $\Ss_{n}^X$ defined as in  \eqref{Ss_{n}^X}--\eqref{S_n_k},
	\begin{equation*} 
			\operatorname{Comp} \brac{\Ss_{n}^X }	
			\ \le  \
			C n \log n.
	\end{equation*} 
	\begin{equation*} 	\label{L_2-rate}
		\begin{split}
			\sup_{v\in B_X^{\alpha,\tau}}		\|v-\Ss^X_n v\|_{L_2(U,X;\mu)} 
		\ \le \	
			\begin{cases}
			C_\tau	n^{-\alpha}(\log n )^\tau \quad &{\rm if }  \ \alpha < 1/q_2,\\
			C_\tau	n^{-\alpha}(\log n )^{1 + \tau}\quad &{\rm if }  \ \alpha = 1/q_2,\\
			C	n^{-\beta}\black{(\log n)^{1 + \beta \tau /\alpha}}\quad  & {\rm if }  \ \alpha > 1/q_2,
			\end{cases}
		\end{split}
	\end{equation*}
\black{where the positive constant $C$ depends only on $\alpha,  q_1,q_2$ and $K$; the positive constants $C_\tau$ depend  only on 
	$\alpha,  q_1,q_2, K,$ and, explicitly, on $\tau$.}
\end{theorem}

\begin{proof}
 Let $v \in B_X^{\alpha,\tau}$	be given. By the assumptions we have 
	\begin{equation} \label{norm{delta_k v}}
\norm{\delta_k v}{H_X^{\alpha,\tau}} \le 2K \ \ \forall k \in \NN.
\end{equation}
By the triangle inequality,
	\begin{equation*} 
		\|v- \Ss_n^X v\|_{L_2(U,X;\mu)}
		\ \le \ 
		\|v- P_{k_n} v\|_{L_2(U,X;\mu)}
		\ + \
		\sum_{k=1}^{k_n}	\|\delta_k v -  S_{n_k}^X (\delta_k v)\|_{L_2(U,X;\mu)}.
	\end{equation*}	
		By  Parseval's identity  and the equality $k_n:= \lfloor \log \xi_n \rfloor$ we deduce that
	\begin{equation} \nonumber
		\begin{split}
			\|v- P_{k_n} v\|_{L_2(U,X;\mu)}^2
			\ &= \
			\sum_{s \in \NN}\|(v- P_{k_n} v)_s\|_{X}^2
			\\
			\ &\le \
			2^{-2\alpha k_n} \brac{k_n}^{2\tau} 2^{2\alpha k_n}\brac{k_n}^{-2\tau}
			\sum_{s \in \NN}\brac{\sigma_{2;s} \|(v- P_{k_n} v)_s\|_{X}}^2
			\\ 
			\ &\le \
			2^{-2\alpha \lfloor \log \xi_n \rfloor}\brac{k_n}^{2\tau}\norm{v}{H_{X,\bsigma_2}^{\alpha}} 
			\ \le \		2^{-2\alpha \lfloor \log \xi_n \rfloor}\brac{\log \xi_n}^{2\tau}.
		\end{split}
	\end{equation}	
	This means that 
	\begin{equation} \label{|v- P_{k_{xi_n}} v|}
			\|v- P_{k_n} v\|_{L_2(U,X;\mu)}
			\ \le \		2^{-\alpha \lfloor \log \xi_n \rfloor}\brac{\log \xi_n}^{\tau}.
	\end{equation}	

Assume  $\alpha \le 1/q_2$. By \eqref{norm{delta_k v}},	
\begin{equation} \label{norm{delta_k v}2}
	\norm{\delta_k v} {H_{X,\bsigma_2}}
	\ \le \
	2K 2^{-\alpha k} k^\tau, \ \ k \in \NN.
\end{equation} 
	By  Lemma \ref{lemma:sampling inequalityX} for
$\bsigma = \bsigma_2$  and  $q=q_2$,
there exist  $\by_{k,1},...,\by_{k,n_k} \in U$ and $h_{k,1},...,h_{k,n_k} \in L_2(U,\CC;\mu)$ such that
\begin{equation*} 
	\begin{split}			
		\|\delta_k v -  S_{n_k}^X (\delta_k v)\|_{L_2(U,X;\mu)}
		\ &\le \ 
		\norm{\delta_k v}{H_{X,\bsigma_2}}	\brac{\frac{1}{n_k} \sum_{j \ge n_k} d_j^p}^{1/p},
	\end{split}
\end{equation*}
where $d_j := d_j(B_{\CC,\bsigma_2},L_2(U,\CC;\mu))$.
Hence, by  Lemma \ref{lemma:d_n><} for 
$\bsigma = \bsigma_2$  and  $q=q_2$,  \eqref{norm{delta_k v}2} and the inequality $p/q_2 > 1$,
we derive for $k \le k_n$,
\begin{equation} \label{ineq3}
	\begin{split}			
		\|\delta_k v -  S_{n_k}^X (\delta_k v)\|_{L_2(U,X;\mu)}
		\ &\le \ 
		C\, 2^{-\alpha k} k^\tau
		\brac{\frac{1}{n_k} \sum_{j \ge n_k} j^{-p/q_2}}^{1/p}
		\ \le \ 
		C^*\, 2^{\brac{1/q_2-\alpha}k}  k^\tau n^{-1/q_2},
	\end{split}
\end{equation}
\black{where $C^*$ is a positive constant depending on $q_2$ and $p$ only.}

For any $n \in \NN$, choose $\xi_n$ as a number satisfying the inequalities 
\begin{equation*} 
	\xi_n 
	\ \le \
	n
	\ < \ 2\xi_n.
\end{equation*}	
\black{With this choice from  the equality $k_n:= \lfloor \log \xi_n \rfloor$,  
\eqref{|v- P_{k_{xi_n}} v|} and \eqref{ineq3} we deduce that}
\begin{equation} \nonumber
	\|v- P_{k_n} v\|_{L_2(U,X;\mu)}
	\ \le \ 
2^{-\alpha \lfloor \log \xi_n \rfloor} (\log \xi_n )^\tau
	\ \le \	2^{\alpha} n^{- \alpha}  (\log n )^\tau,
\end{equation}	
and 
 for  $\alpha < 1/q_2$,
\begin{equation} \nonumber
	\sum_{k=1}^{k_n} 	\|\delta_k v -  \black{S_{n_k}^X} (\delta_k v)\|_{L_2(U,X;\mu)}
	\ \le \ C^*\, n^{- 1/q_2}
	\sum_{k=1}^{k_n}   2^{\brac{1/q_2-\alpha}k}k ^\tau
	\ \le \ \black{C'_\tau \,n^{- \alpha},}
\end{equation}
\black{with
\begin{equation*} 
C'_\tau:= C^* \sum_{k=1}^\infty   2^{-\brac{\alpha - 1/q_2}k}k ^\tau,
\end{equation*}
}
and for  $\alpha = 1/q_2$,
\begin{equation} \nonumber
	\begin{split}
	\sum_{k=1}^{k_n} 	\|\delta_k v -  \black{S_{n_k}^X}  (\delta_k v)\|_{L_2(U,X;\mu)}
	\ &\le \ C^*\, n^{- 1/q_2}
	\sum_{k=1}^{k_n}  k ^\tau \\
	\ &\black{\le \ C^* \,n^{- \alpha}\int_0^{k_n} x^\tau \rd x
	\ \le \ C^{''}_\tau \,n^{- \alpha}(\log n )^{1 + \tau},}			
	\end{split}
\end{equation}
\black{with
\begin{equation*} 
	C^{''}_\tau:= C^* (\tau +1)^{-1}.
\end{equation*}
}
Summing up, we arrive at the inequalities
\begin{equation} \label{C_tau:=(1)}
	\begin{split}
		\|v - \Ss_n^X v\|_{L_2(U,X;\mu)}
		\ &\le \
		\black{C_\tau} \, n^{- \alpha}(\log n )^\tau \ \ 
		\black{\text{with} 
		\ \ C_\tau:= \max\brac{2^\alpha, C'_\tau} }\ \ \text{if} \ \ \alpha < 1/q_2,
	\end{split}
\end{equation}	
and 
\begin{equation} \label{C_tau:=(2)}
	\begin{split}
		\|v - \Ss_n^X v\|_{L_2(U,X;\mu)}
		\ &\le \
		\black{C_\tau} \, n^{- \alpha}(\log n )^{1 + \tau}  
		\ \  	\black{\text{with}  \ \ 
		C_\tau:= \max\brac{2^\alpha, C^{''}_\tau}}  
		\ \ \text{if} \ \  \alpha = 1/q_2.
	\end{split}
\end{equation}	

For the computational complexity $\operatorname{Comp} \brac{\Ss_{n}^X v}$  we have that
\begin{equation*} 
	\begin{split}
		\operatorname{Comp} \brac{\Ss_{n}^X v}	
		\ \le 
		2\sum_{k=1}^{k_n}  2^kn_k
		\ & \le  \ 
		2 \sum_{k=1}^{k_n}  2^k (\xi_n 2^{-k})
		\ =  \
2 \xi_n	\sum_{k=1}^{k_n}  1
		\ \le  \
	 C n \log n.
	\end{split}
\end{equation*}

 	 We now consider the case $ \alpha > 1/q_2 $.  In this case, we have $\alpha > \beta$. For any $n \in \NN$, choose $\xi_n$ as a number satisfying the inequalities 
 	 \begin{equation*} 
 	 	\xi_n^{q_1(\alpha + \delta)} \ \le \ n \ < \ 2\, \xi_n^{q_1(\alpha + \delta)}.
 	 \end{equation*}	 
\black{ From the  last inequalities and}
\eqref{|v- P_{k_{xi_n}} v|}  we deduce that
 \begin{equation} \nonumber
 	\|v- P_{k_n} v\|_{L_2(U,X;\mu)}
 	\ \le \
 	2^{-\alpha \lfloor \log \xi_n \rfloor}\brac{\log \xi_n}^{\tau}
 	\ \le \	
 C\,  n^{- \beta}\brac{\log  n}^{\tau}.
 \end{equation}	
 	By  Lemma \ref{lemma:d_n(delta_kB_X)},
 there exist  $\by_{k,1},...,\by_{k,n_k} \in \black{U_0}$, $h_{k,1},...,h_{k,n_k} \in L_2(U,\CC;\mu)$ such that
 \begin{equation*} 
 	\begin{split}			
 		\|\delta_k v -  S_{n_k}^X (\delta_k v)\|_{L_2(U,X;\mu)}
 		\ &\le \ 
 		\norm{\black{v}}{H_X^{\alpha,\tau}}	\brac{\frac{1}{n_k} \sum_{j \ge n_k} d_j^p}^{1/p},
 	\end{split}
 \end{equation*}
 where $d_j := d_j(\delta_k B_\CC^{\alpha,\tau},L_2(U,\CC;\mu))$.
 Hence, by Corollary \ref{corollary:d_n(delta_kB_CC)} and the inequality $p \beta >1$,
 we derive for $k \le k_n$,
 \begin{equation*} 
 	\begin{split}			
 		\|\delta_k v -  S_{n_k}^X (\delta_k v)\|_{L_2(U,X;\mu)}
 		\ &\le \ 
 		C\, 
 	(2^k n_k)^{-\beta} (k + \log n_k)^{\beta \tau /\alpha}
 			\ \le \
 	C\, n^{-\beta}(\log n)^{\beta \tau /\alpha}.
 	\end{split}
 \end{equation*}
 Hence,   we have
 \begin{equation} \nonumber
 	\sum_{k=1}^{k_n} 	\|\delta_k v -  \black{S_{n_k}^X} (\delta_k v)\|_{L_2(U,X;\mu)}
 	\ \le \
 	\sum_{k=1}^{k_n} C \,  n^{- \beta}(\log n)^{\beta \tau /\alpha}
 	\ \le \ C \,n^{- \beta} (\log n)^{1 + \beta \tau /\alpha},
 \end{equation}
 Summing up, we arrive at the inequality
 \begin{equation} \nonumber
 	\begin{split}
 		\|v - \Ss_n^X v\|_{L_2(U,X;\mu)}
 		\ &\le \
 		C \, n^{- \beta} \black{(\log n)^{1 + \beta \tau /\alpha}}, \ \ \text{if} \ \ \alpha  >  1/q_2.  
 	\end{split}
 \end{equation}	
For the computational complexity $\operatorname{Comp} \brac{\Ss_{n}^X v}$  we have that
\begin{equation*} 
	\begin{split}
		\operatorname{Comp} \brac{\Ss_{n}^X v}	
		\ \le 
		2\sum_{k=1}^{k_n}  2^kn_k
		\ & \le  \ 
		2 \sum_{k=1}^{k_n} 2^k \brac{ n 2^{-k}}
		\ =  \
		2 n	\sum_{k=1}^{k_n}  1
		\ \le  \
		C n \log n.
	\end{split}
\end{equation*}  	 
	\hfill
\end{proof}

\black{
Note that for a fixed $\varepsilon >0$,
\begin{equation} \label{emb-ineq}
\norm{v}{H_X^{\alpha}}
\le 
\bar{C}_\varepsilon\norm{v}{H_X^{\alpha, 1/2+\varepsilon}}, \quad 
\bar{C}_\varepsilon:= \brac{\sum_{k \in \NN} k^{-1-2\varepsilon}}^{1/2} <\infty.
\end{equation} 
Hence,  Theorem \ref{thm:sampling} can be reformulated more conveniently  for the space $ H_X^{\alpha}$ as follows.
}

  \begin{theorem}\label{thm:sampling2}
  	Let $0 < q_1 \le q_2 <2$ and $\varepsilon$ be any \black{fixed} positive number. Let the number $\beta$ be defined as in \eqref{delta, beta}.
 Let $(\xi_{n})_{n \in \NN}$ be a sequence of increasing positive numbers $\xi_n$ satisfying the condition \eqref{xi_n}. 
  	Assume that there is a constant \black{$K$} such that for every $v \in H_X^{\alpha,1+\varepsilon}$ and every $m \in \NN$, 
  	$\norm{P_m v}{H_X^{\alpha, 1+\varepsilon}} \le \black{K \norm{v}{H_X^{\alpha,1+\varepsilon}}}$. 
  	Then for all $n \ge 2$ there exist  $\by_{k,1},...,\by_{k,n_k} \in \black{U_0}$ and $h_{k,1},...,h_{k,n_k} \in L_2(U,\CC;\mu)$, $k=0,...,k_{\lceil n/\log n \rceil}$,  such that for the operator $\bar{\Ss}^X_n:= \Ss_{\lceil n/\log n \rceil}^X$ defined as in \eqref{Ss_{n}^X}--\eqref{S_n_k}, 
  	\begin{equation*} \label{comp1}
  		\operatorname{Comp} \brac{\bar{\Ss}^X_n}	
  		\ \le  \
  		 n,
  	\end{equation*} 
  	 and
  	\begin{equation*} 	\label{L_2-rate}
  		\begin{split}
  			\sup_{v\in B_X^{\alpha}}		\|v-\bar{\Ss}^X_n v\|_{L_2(U,X;\mu)} 
  			\ \le \
  			\begin{cases}
  			C_\varepsilon	n^{-\alpha}(\log n )^{\alpha + \black{1/2}+\varepsilon}\quad &{\rm if }  \ \alpha < 1/q_2,\\
  			C_\varepsilon	n^{-\alpha}(\log n )^{\alpha + \black{3/2} +\varepsilon}\quad &{\rm if }  \ \alpha = 1/q_2, \\
  			\black{C_\varepsilon}	n^{-\beta}
  			(\log n )^{\black{1 + \beta (1 + 1/2 \alpha )+ \varepsilon\beta/\alpha}}\ \ 
  			\quad  &{\rm if }  \ \alpha > 1/q_2,
  			\end{cases}
  		\end{split}
  	\end{equation*}
  	\black{ the positive constants $C_\varepsilon$ depends  only on 
  		$\alpha,  q_1,q_2, K,$ and, explicitly, on $\varepsilon$.}
  \end{theorem}
  
  \black{
  	\begin{remark} \label{rmk}
  		{\rm 
  			We remark on how the constants in Theorems \ref{thm:sampling} and \ref{thm:sampling2} depend on the relevant parameters.
  			The dependence of the constants $C_\tau$ in the cases $\alpha < 1/q_2$ and $\alpha = 1/q_2$ in Theorem~\ref{thm:sampling} are given explicitly in \eqref{C_tau:=(1)} and \eqref{C_tau:=(2)}, respectively. In contrast,  the constant $C$ in the case $\alpha > 1/q_2$ depends on $\alpha,  q_1,q_2$ and $K$ only. 	
  			In Theorem \ref{thm:sampling2}, for all the cases  the positive constants $C_\varepsilon$ depend  only on 
  			$\alpha,  q_1,q_2, K,$ and, explicitly, on $\varepsilon$ due to 
  			\eqref{C_tau:=(1)} -- \eqref{emb-ineq}.  			 
  			Similar remarks regarding the constants in Theorems \ref{thm:sampling-modification(1)} and \ref{thm:sampling-modification(2)} apply.
  		}
  	\end{remark}
  }
   
Next, we apply  Theorem \ref{thm:sampling2}
 to Bochner spaces with infinite tensor-product standard Gaussian  measure relevant to the applications  to  holomorphic functions and solutions to parametric PDEs with random inputs  in Sections 
 \ref{Applications to  holomorphic functions} and \ref{Applications to  parametric PDEs}, respectively.

 We recall a concept of
 standard Gaussian measure $\gamma(\by)$ on $\RRi$ as 
 the infinite tensor product of copies of the one-dimensional standard Gaussian
 measure $\gamma(y_i)$:
 \begin{equation} \nonumber
 	\gamma(\by) 
 	:= \ 
 	\bigotimes_{j \in \NN} \gamma(y_j) , \quad \by = (y_j)_{j \in \NN} \in \RRi.
 \end{equation}
 (The sigma algebra for $\gamma(\by)$ is generated by the set of cylinders $A:= \prod_{j \in \NN} A_j$, where $A_j \subset \RRR$ are univariate $\gamma$-measurable sets and only a finite number of $A_i$ are different from $\RRR$. For such a set $A$, we have $\gamma(A) = \prod_{j \in \NN} \gamma(A_j)$).
 
 Let $X$ be a separable Hilbert space. Then a function $v \in L_2(\RRi,X;\gamma)$  can be represented  by the Hermite GPC expansion
 \begin{equation} \label{GPCexpansion}
 	v=\sum_{\bs\in\FF} v_\bs \,H_\bs, \quad v_\bs \in X,
 \end{equation}
 with
 \begin{equation*}
 	H_\bs(\by)=\bigotimes_{j \in \NN} H_{s_j}(y_j),\quad 
 	v_\bs:=\int_{\RRi} v(\by)\,H_\bs(\by)\, \rd\gamma (\by), \quad 
 	\bs \in \FF.
 \end{equation*}
 Here $(H_s)_{s \in \NN_0}$  are the univariate orthonormal Hermite polynomials, $\FF$ is the set of all sequences of non-negative integers $\bs=(s_j)_{j \in \NN}$ such that their support 
 $\supp (\bs):= \{j \in \NN: s_j >0\}$ is a finite set.
 Notice that the complex-valued  $(H_\bs)_{\bs \in \FF}$ are an orthonormal basis of $L_2(\RRi,\CC;\gamma)$. 
 Moreover, for every $v \in L_2(\RRi,X;\gamma)$  represented by the 
 series \eqref{GPCexpansion},  Parseval's identity holds
 \begin{equation} \nonumber
 	\|v\|_{L_2(\RRi,X;\gamma)}^2
 	\ = \ \sum_{\bs\in\FF} \|v_\bs\|_X^2.
 \end{equation}

 Let $\bsigma=\brac{\sigma_{\bs}}_{\bs \in \FF}$ be a set of positive numbers strictly larger than $1$ such that 
 $\bsigma^{-1}:=\brac{\sigma_{\bs}^{-1}}_{\bs \in \FF} \in \ell_2(\FF)$. 
 Denote by $H_{X,\bsigma}$ the linear subspace in $L_2(\RRi,X;\gamma)$ of all $v$ such that the norm
 \begin{equation} \nonumber
 	\|v\|_{H_{X,\bsigma}}
 	:= \
 	\brac{\sum_{\bs \in \FF} \brac{\sigma_{\bs} \|v_\bs\|_X}^2}^{1/2} < \infty.		
 \end{equation}
 In particular, the space $H_{\CC,\bsigma}$ is the linear subspace in $L_2(\RRi,\CC;\gamma)$ equipped with its own inner product
 \begin{equation} \nonumber
 	\langle f,g \rangle_{H_{\CC,\bsigma}}
 	:= \
 	\sum_{\bs \in \FF} \sigma_{\bs}^2 
 	\langle f,H_\bs \rangle_{L_2(\RRi,\CC;\gamma)}
 	\overline{\langle g,H_\bs \rangle_{L_2(\RRi,\CC;\gamma)}}.
 \end{equation}
 The space $H_{\CC,\bsigma}$ is a reproducing kernel Hilbert space with the reproducing kernel
 \begin{equation} \nonumber
 	K(\cdot,\by)
 	:= \
 	\sum_{\bs \in \FF} \sigma_{\bs}^{-2} H_\bs (\cdot)\overline{H_\bs (\by)}
 \end{equation}
 with the eigenfunctions $\brac{H_\bs}_{\bs \in \FF}$ and the eigenvalues 
 $\brac{\sigma_{\bs}^{-1}}_{\bs \in \FF}$. Moreover, $K(\bx,\by)$ satisfies the finite trace assumption
 \begin{equation*}
 	\int_{\RRi} K(\bx,\bx) \rd \gamma(\bx) \ < \ \infty.
 \end{equation*}
 
 \black{ Notice that if $\|\bsigma^{-1}\|_{\ell_q(\FF)} < \infty$ for some $0 <q \le 2$, then for every $v \in B_{X,\bsigma}$, the series \eqref{GPCexpansion} converges absolutely and unconditionally in 
 	$L_2(\RRi,X;\gamma)$ to $v$ (see \cite[Lemma 3.1]{Dung22}).
 	Hence,  we can reorder the countable set $\FF$ as $\FF = (\bs_j)_{j \in \NN}$ so that the sequence 
 	$\bsigma= (\sigma_{\bs_j})_{j \in \NN}$ is non-decreasing. Put 
 	 $\sigma_j:=\sigma_{\bs_j}$, \ $\varphi_j:=\phi_{\bs_j}$ and $v_j:=v_{\bs_j}$.  Then  $B_{X,\bsigma}$ can be seen as the set of all functions $v \in L_2(U,X;\gamma))$ for $U:=\RRi$, represented by 
 	the series 
 	\begin{equation*} 
 		v= \sum_{j \in \NN}  v_j\,\varphi_j, \quad v_j \in X,
 	\end{equation*}
 	such that
 	\begin{equation*} 
 		\left(\sum_{j \in \NN} (\sigma_j \|v_j\|_{X})^2\right)^{1/2} \ \le 1.
 	\end{equation*}
 }
 This means that the pointwise evaluations $f(\by)$ are well-defined for every $\by$ belonging the full measure subset $U_0$ of $U$.  Let $\brac{\psi_k}_{k \in \NN}$  be an orthonormal basis of  $X$. 
 If $v \in H_{X,\bsigma}$, then $\langle v,\psi_k \rangle_X \in H_{\CC,\bsigma}$ for every $k$. 
 Consequently, the pointwise evaluations $\langle v(\by),\psi_k \rangle_X $ are well-defined for every
 $\by \in U_0$. This implies that the pointwise evaluations  $v(\by)$ are also well-defined for every
 $\by \in U_0$.  
 
  Let $(V_k)_{k \in \NN}$ be a given sequence of subspaces $V_k \subset \black{L_2(\RRi,X;\gamma)}$ of dimension $2^k$, and $(P_k)_{k \in \NN}$  a given sequence of uniformly bounded linear projectors  $P_k$ from $\black{L_2(\RRi,X;\gamma)}$ onto $V_k$. 
For $i=1,2$,  let $0< q_i \le 2$ and   $\bsigma_i:=(\sigma_{i;\bs})_{\bs \in \FF}$ be given sets of numbers strictly larger than $1$, such that $\bsigma_i^{-1}:=\brac{\sigma_{i,\bs}^{-1}}_{\bs \in \FF} \in \ell_{q_i}(\FF)$. 
 For   given numbers $\alpha> 0$ and  $\tau > 0$,  the linear subspaces  $ H_X^{\alpha}$ and $ H_X^{\alpha,\tau}$ in $L_2(\RRi,X;\gamma)$  are defined in the same manner as \eqref{H_{X}^{alpha}} and \eqref{H_{X}^{alpha,tau}}, respectively.
 
 \black{Let $\RRi_0 \subset \RRi$ be a set of full measure such that if $v \in H_X^\alpha$,  the pointwise evaluations  $v(\by)$ are also well-defined for every
 $\by \in \RRi_0$.}   Theorem \ref{thm:sampling2} for the space $L_2(\RRi,X;\gamma)$ is read as 
 \begin{theorem}\label{thm:samplingDD}
 	Let the assumptions of Theorem \ref{thm:sampling2} hold for the space $L_2(\RRi,X;\gamma)$.
 	Then for all $n \ge 2$ there exist  $\by_{k,1},...,\by_{k,n_k} \in \black{\RRi_0}$ and $h_{k,1},...,h_{k,n_k} \in L_2(\RRi,\CC;\gamma)$, $k=0,...,k_{\lceil n/\log n \rceil}$,  such that for the operator $\bar{\Ss}^X_n:= \Ss_{\lceil n/\log n \rceil}^X$ defined as in \eqref{Ss_{n}^X}--\eqref{S_n_k}, 
 	\begin{equation*} \label{comp1}
 		\operatorname{Comp} \brac{\bar{\Ss}^X_n}	
 		\ \le  \
 		n,
 	\end{equation*} 
 	and
 	\begin{equation} 	\label{L_2-rate2}
 			\sup_{v\in B_X^{\alpha}}		\|v-\bar{\Ss}^X_n v\|_{L_2(\RRi,X;\gamma)} 
 			\ \le \
 			\begin{cases}
 				C_\varepsilon	n^{-\alpha}(\log n )^{\alpha + \black{1/2}+\varepsilon}\quad &{\rm if }  \ \alpha < 1/q_2,\\
 				C_\varepsilon	n^{-\alpha}(\log n )^{\alpha + \black{3/2} +\varepsilon}\quad &{\rm if }  \ \alpha = 1/q_2, \\
 				\black{C_\varepsilon}	n^{-\beta}
 				(\log n )^{\black{1 + \beta (1 + 1/2 \alpha )+ \varepsilon\beta/\alpha}}\ \ 
 				\quad  &{\rm if }  \ \alpha > 1/q_2,
 			\end{cases}
 	\end{equation}
 		\black{where  the positive constants $C_\varepsilon$ depend  only on 
 		$\alpha,  q_1,q_2, K,$ and, explicitly, on $\varepsilon$.}
 	\end{theorem}

 	\section{Applications to  holomorphic functions} 
\label{Applications to  holomorphic functions}

In this section, we apply the results on  multi-level linear sampling recovery in the space $L_2(\RRi,X;\gamma)$ to holomorphic functions. The results of this application will imply convergence rates of fully discrete multi-level simultaneous   linear collocation approximation of solutions to parametric elliptic PDEs \eqref{parametricPDE} on bounded polygonal domain $D$ with log-normal random inputs \eqref{lognormal}, based on a finite number of their values at points in the spatial-parametric domain.

We recall  a concept of  $(\bb,\xi,\delta,X)$-holomorphic
functions  which has been introduced in \cite[Definition 4.1]{DNSZ2023}. 
For $m\in\NN$ and a positive sequence $\bvarrho=(\varrho_j)_{j=1}^m$,  we put
\begin{equation*}
	\label{eq:Sjrho}
	\Ss(\bvarrho) := \set{\bz\in \CC^m}{|\mathfrak{Im}z_j| < \varrho_j~\forall j}\qquad\text{and}\qquad
	\Bb(\bvarrho) := \set{\bz\in\CC^m}{|z_j|<\varrho_j~\forall j}.
\end{equation*}

Let $X$ be a complex separable Hilbert space,
$\bb=(b_j)_{j\in\NN}$ a positive sequence, and $\xi>0$, $\delta > 0$.
For $m\in\NN$ we say that a positive sequence $\bvarrho=(\varrho_j)_{j=1}^m$ is
\emph{$(\bb,\xi)$-admissible} if
\begin{equation*}\label{eq:adm}
	\sum_{j=1}^m b_j\varrho_j\leq \xi\,.
\end{equation*}
A function $v\in L_2(\RRi,X;\gamma)$ is called
$(\bb,\xi,\delta,X)$-holomorphic if
\begin{enumerate}
	\item[{\rm (i)}]\label{item:hol} for every $m\in\NN$ there exists
	$v_m:\RR^m\to X$, which, for every $(\bb,\xi)$-admissible
	$\bvarrho$, admits a holomorphic extension
	(denoted again by $v_m$) from $\Ss(\bvarrho)\to X$; furthermore,
	for all $m<m'$
	\begin{equation*}\label{eq:un=um}
		v_m(y_1,\dots,y_m)=v_{m'}(y_1,\dots,y_m,0,\dots,0)\qquad\forall (y_j)_{j=1}^m\in\RR^m,
	\end{equation*}
	
	\item[{\rm (ii)}]\label{item:varphi} for every $m\in\NN$ there exists
	$\varphi_m:\RR^m\to\RR_+$ such that
	$\norm{\varphi_m}{L^2(\RR^m;\gamma)}\le\delta$ and
	\begin{equation*} \label{ineq[phi]}
		\sup_{\text{\black{$\bvarrho$} is $(\bb,\xi)$-adm.}}~\sup_{\bz\in
			\Bb(\bvarrho)}\norm{v_m(\by+\bz)}{X}\le
		\varphi_m(\by)\qquad\forall\by\in\RR^m,
	\end{equation*}
	\item[{\rm (iii)}]\label{item:vN} with $\tilde v_m:\RRRi\to X$ defined by
	$\tilde v_m(\by) :=v_m(y_1,\dots,y_m)$ for $\by\in \RRRi$ it holds
	\begin{equation*}
		\lim_{m\to\infty}\norm{v-\tilde v_m}{L_2(\RRi,X;\gamma)}=0.
	\end{equation*}
\end{enumerate}

The following key result on  weighted $\ell_2$-summability of $(\bb,\xi,\delta,X)$-holomorphic functions has been proven in \cite[\black{Theorem} 4.9]{DNSZ2023}.

\begin{lemma} \label{lemma:weighted summability, holomorphic} Let $v$ be
	$(\bb,\xi,\delta,X)$-holomorphic for some $\bb\in \ell_p(\NN)$ with $0< p <1$.  Let $\eta\in\NN_0$ and let the sequence $\brho=(\rho_j)_{j \in \NN}$ be defined by
	$$
\black{	\rho_j:=b_j^{p-1}\frac{\xi}{4\sqrt{\eta!} \norm{\bb}{\ell_p(\NN)}}.}
	$$
	\black{Assume that $\bb$ is a non-increasing sequence.}
	Then we have 
		\begin{equation*} \label{ell_2-summability}
			\left(\sum_{\bs\in\FF} (\sigma_\bs \|v_\bs\|_{X})^2\right)^{1/2} \ \le M \ <\infty, \ \ \text{with} \ \
			\norm{\bsigma^{-1}}{\ell_q(\NN)} \le N < \infty,
		\end{equation*}	
		where \black{$q := p/(1-p)$,} 
		the set $\bsigma:=(\sigma_\bs)_{\bs \in \FF}$ is defined by
\begin{equation} \label{sigma_s}
	\sigma_{\bs}^2:=\sum_{|\bs'|_{\infty}\leq \eta}{\bs\choose \bs'} \prod_{j \in \NN}\rho_j^{2s_j'},
\end{equation}	
\black{with $|\bs'|_{\infty}:= \sup_{j \in \NN} s'_j$, and} 
		$M= \delta C_{\bb,\xi,\eta}$ with some positive  constant $C_{\bb,\xi,\eta}$, and
		$N=  K_{\bb,\xi,\eta}$.
\end{lemma}

 This theorem allows us to apply weighted $\ell_2$-summability for collocation approximation of solutions  $u(\by)$  as $(\bb,\xi,\delta,X)$-holomorphic functions on various function spaces $X$, to a wide range of parametric and stochastic PDEs with log-normal inputs, specially,  Kondrat'ev spaces $X= K^r_\varkappa(D)$ for  the parametric elliptic PDEs \eqref{ellip} on bounded polygonal domain with log-normal inputs \eqref{lognormal} a detailed analysis of which will be presented in Section \ref{Applications to  parametric PDEs}.  For more information see \cite{DNSZ2023}.

For  a function $v$ defined on $\RRi$ taking values in a separable Hilbert space $X$, we say that $v$ satisfies Assumption~\ref{ass:ml} if

\begin{assumption}\label{ass:ml}
	$\alpha > 0$,  $\eta \in \NN$, \black{$0<p_1\le p_2< 1$,} $\bb_i\in\ell_{p_1}(\NN)$,
	$i=1,2$, are given sequences, $\xi>0$, $\delta>0$. \black{The sequences $\brho_i=(\rho_{i;j})_{j \in \NN}$, $i=1,2$, are defined by
\begin{equation} \label{rho_{i,j}}
	\rho_{i;j}:=	b_{i;j}^{p_i-1}\frac{\xi}{4\sqrt{\eta!} \norm{\bb_i}{\ell_{p_j}(\NN)}},  \ j\in \NN,
\end{equation}	
}
	  \black{$(V_k)_{k \in \NN}$}  is a given sequence of subspaces $V_k \subset L_2(\RRi,X;\gamma)$ of dimension $2^k$, and $\Pp =(P_k)_{k \in \NN}$  a given sequence of linear uniformly bounded projectors  from $L_2(\RRi,X;\gamma)$ onto $V_k$ such that
	\begin{enumerate}
		\item $v\in L_2(\RRi,X;\gamma)$ and is $(\bb_1,\xi,\delta,X)$-holomorphic,
		\item\label{item:u-ujbound} $(v- P_k v)\in L_2(\RRi,X;\gamma)$ and is
		$(\bb_1,\xi,\delta,X)$-holomorphic for every $k\in\NN_0$,
		\item\label{item:u-ujclose} $(v-P_k v)$ is
		$(\bb_2,\xi,\delta 2^{-{\alpha k}},X)$-holomorphic for every
		$k\in\NN_0$,
		\item\label{propertiesP_m}
		For every $m \in \NN$ and every linear bounded operator $Q$ in $L_2(\RRi,X;\gamma)$ with $\|Q\| \le K$, the function 
		$w:=Q v$ possesses the properties  (i)--(iii) with the parameter $\delta$ replaced by $K\delta$, the same parameters 
		$\alpha, \eta, p_1, p_2, \xi$ and sequences $\bb_1, \bb_2, \Vv, \Pp$.
	\end{enumerate}
\end{assumption}

For a space $H_X^{\alpha}$ with $\norm{\bsigma_i^{-1}}{\ell_q(\NN)} \le N$, $i=1,2$, denote
\begin{equation*} 
	B_X^{\alpha}(M,N):= \brab{v \in H_X^{\alpha,\tau}: \, \norm{v}{H_X^{\alpha,\tau}} \le M}.
\end{equation*}

\begin{lemma} \label{lemma:B_X, holomorphic} Let $H_X^{\alpha}$ and $H_X^{\alpha,\tau}$ with $\tau > 1$ be defined so that  for  $i=1,2$, $q_i := p_i/(1-p_i)$ and
	$\bsigma_i:=(\sigma_{i,\bs})_{\bs \in \FF}$ be given by \eqref{sigma_s} for $\brho_i$ as in \eqref{rho_{i,j}}, $\norm{\bsigma_i^{-1}}{\ell_{q_i}(\FF)} \le N$, respectively.  Then every function $v$ satisfying Assumption~\ref{ass:ml} belongs to 
	$	B_X^{\alpha}(M,N)$ for  $M= \delta C_{\bb,\xi,\eta}$ and $N=  K_{\bb,\xi,\eta,\tau}$ with some positive  constants $C_{\bb,\xi,\eta}$, and $K_{\bb,\xi,\eta}$. Moreover,
	$$
	\norm{P_m v}{H_X^{\alpha,\tau}} \le K_{\bb,\xi,\eta,\tau} M\norm{v}{H_X^{\alpha,\tau}}, \ \ m \in \NN.
	$$
\end{lemma}

\begin{proof}	
Assume that $v$ satisfies Assumption~\ref{ass:ml}. 	
 Then 
we have by Lemma \ref{lemma:weighted summability, holomorphic},
\begin{equation*} \label{ell_2-summability1}
	\left(\sum_{\bs\in\FF} (\sigma_{1;\bs} \|v_\bs\|_{X})^2\right)^{1/2} \ \le M, \ \
\text{with} \ \
\norm{\bsigma_1^{-1}}{\ell_q(\FF)} \le N < \infty,
\end{equation*}	
and  for every $k \in \NN$,
\begin{equation*} \label{ell_2-summability2}
	\left(\sum_{\bs\in\FF} (\sigma_{2;\bs} \|(v - P_k v)_\bs\|_{X})^2\right)^{1/2} \ \le M 2^{-\alpha k}\ <\infty, \ \ 
	\text{with} \ \
	\norm{\bsigma_2^{-1}}{\ell_q(\FF)} \le N < \infty,
\end{equation*}	
and in addition, for  \black{every} $m \in \NN$, 
\begin{equation*} \label{ell_2-summability1}
	\left(\sum_{\bs\in\FF} (\sigma_{1;\bs} \|\black{(P_m v)_\bs}\|_{X})^2\right)^{1/2} \ \le \ K_{\bb,\xi,\eta} \black{M} \ \
	\left(\sum_{\bs\in\FF} (\sigma_{2;\bs} \|\black{(P_m v- P_k (P_m v))_\bs}\|_{X})^2\right)^{1/2} 
	\ \le \ K_{\bb,\xi,\eta}M 2^{-\alpha k}.
\end{equation*}			
Hence,  $v \in B_X^{\alpha}(M,N)$.  Moreover,
$\norm{P_m v}{H_X^{\alpha,\tau}} \le K_{\bb,\xi,\eta,\tau}M\norm{v}{H_X^{\alpha,\tau}}$,  
$m \in \NN$, since  $\tau > 1$. 
	\hfill
\end{proof}

By applying Theorem \ref{thm:samplingDD}, from Lemma \ref{lemma:B_X, holomorphic}  we obtain

\begin{theorem}\label{thm:sampling-log-normal-holomorphicK} 	Let $v$ satisfy Assumption \ref{ass:ml}. 
\black{Let  $\varepsilon$ be any fixed positive number.}
Let the number $\beta$ be defined as in \eqref{delta, beta}. 
\black{Let the numbers $\beta$ and $\delta$ be defined as 
\begin{equation} 	\label{delta, beta-hol}
	\beta := \brac{\black{\frac 1 {p_1} - 1}}\frac{\alpha}{\alpha + \delta}, \quad 
	\delta := \frac 1 {p_1} - \frac 1 {p_2}.
\end{equation}
}	
 Let $(\xi_{n})_{n \in \NN}$ be a sequence of increasing positive numbers $\xi_n$ satisfying the condition \eqref{xi_n}.
	Then for all $n \ge 2$ there exist  $\by_{k,1},...,\by_{k,n_k} \in \black{\RRi_0}$ and $h_{k,1},...,h_{k,n_k} \in L_2(\RRi,\CC;\gamma)$, $k=0,...,k_{\lceil n/\log n \rceil}$,  such that for the operator $\bar{\Ss}^X_n:= \Ss_{\lceil n/\log n \rceil}^X$ defined as in \eqref{Ss_{n}^X}--\eqref{S_n_k}, 
	\begin{equation*} \label{comp1}
		\operatorname{Comp} \brac{\bar{\Ss}^X_n}	
		\ \le  \
		n,
	\end{equation*} 
	and 
	\begin{equation*} 	
		\begin{split}
			\|v-\bar{\Ss}^X_n v\|_{L_2(\RRi,X;\gamma)} 
			\ \le \
		\begin{cases}
			C_\varepsilon	n^{-\alpha}(\log n )^{\alpha + \black{1/2}+\varepsilon}\quad &{\rm if }  \ \alpha < \black{1/p_2 - 1},\\
			C_\varepsilon	n^{-\alpha}(\log n )^{\alpha + \black{3/2} +\varepsilon}\quad &{\rm if }  \ \alpha = \black{1/p_2 - 1}, \\
			\black{C_\varepsilon}	n^{-\beta}
			(\log n )^{\black{1 + \beta (1 + 1/2 \alpha )+ \varepsilon\beta/\alpha}}\ \ 
			\quad  &{\rm if }  \ \alpha > \black{1/p_2 - 1},
		\end{cases}	
		\end{split}
	\end{equation*}
		\black{where  the positive constants $C_\varepsilon$ depend  only on 
		$\alpha,  q_1,q_2, K,$ and, explicitly, on $\varepsilon$.}
	\end{theorem}	
	
\section{Multi-level sparse-grid interpolation algorithms}
\label{Multi-level sparse-grid interpolation algorithms}

In this section, we construct fully discrete multi-level sparse-grid  sampling algorithms
for the $(\bb,\xi,\delta,X)$-holomorphic functions satisfying Assumption \ref{ass:ml} except the item \black{(iv)}, based on GPC Lagrange-Hermite interpolation, and prove convergence rates of the approximation by them. It \black{turns} out that in the case $\alpha \le  1/q_2 - 1/2$,
these sampling algorithms give a convergence  rate better than the extended least squares sampling 
algorithm in Theorem \ref{thm:sampling-log-normal-holomorphicK}. The construction and techniques used in this section are a modification of those in \black{\cite{Dung2025-v14}}.

For $m \in \NN$, let $Y_m = (y_{m;k})_{k \in \pi_m}$ be the increasing sequence of  the $m+1$ roots of the Hermite polynomial $H_{m+1}$, ordered as
$$
y_{m,-j} < \cdots < y_{m,-1} < y_{m,0} = 0 < y_{m,1} < \cdots < y_{m,j} \quad {\rm if} \  m = 2j,
$$
$$
y_{m,-j} < \cdots < y_{m,-1} < y_{m,1} < \cdots < y_{m,j} \quad {\rm if} \  m = 2j - 1,
$$
where 
$$
\pi_m:= 
\begin{cases}
	\{-j,-j+1,..., -1, 0, 1, ...,j-1,j \} \ & \ \text{if} \  m = 2j; \\
	\{-j,-j+1,...,-1, 1,...,j-1,j \} \ & \ \text{if} \  m = 2j-1.
\end{cases}
$$
(in particular, $Y_0 = (y_{0;0})$ with $y_{0;0} = 0$).

For  a function $v$ on $\RR$, taking values in a Hilbert space $X$ and $m \in \NN$, we define the   Lagrange interpolation operator $I_m$  by
\begin{equation*} 
	I_m(v):= \ \sum_{k\in \pi_m} v(y_{m;k}) L_{m;k}, \quad 
	L_{m;k}(y) := \prod_{j \in \pi_m \ j\not=k}\frac{y - y_{m;j}}{y_{n;k} - y_{m;j}},
\end{equation*}		
(in particular, $I_0(v) = v(y_{0,0})L_{0,0}(y)= v(0)$ and $L_{0,0}(y)=1$). Notice that $I_m(v)$ is a function on $\RR$ taking values in $X$ and  interpolating $v$ at $y_{m;k}$, i.e., $I_m(v)(y_{m;k}) = v(y_{m;k})$.   

%

We define the univariate operator $\Delta_m$ for $m \in \NN$ by
\begin{equation} \nonumber
	\Delta_m
	:= \
	I_m - I_{m-1},
\end{equation} 
with the convention $I_{-1} = 0$. 

%


For  a function $v$ on $\RRi$, taking values in a Hilbert space $X$, we introduce the tensor product operator $\Delta_\bs$, $\bs \in \FF$, by
\begin{equation*} 
	\Delta_\bs(v)
	:= \
	\bigotimes_{j \in \NN} \Delta_{s_j}(v),
\end{equation*}
where the univariate operator
$\Delta_{s_j}$ is successively applied to the univariate function $\bigotimes_{i<j} \Delta_{s_i}(v)$ by considering it as a 
function of  variable $y_j$ with the other variables held fixed.
We define for $\bs \in \FF$,
\begin{equation} \nonumber
	I_\bs(v)
	:= \
	\bigotimes_{j \in \NN} I_{s_j}(v), \quad
	L_{\bs;\bk}
	:= \
	\bigotimes_{j \in \NN} L_{s_j;k_j}, \quad
	\pi_\bs
	:= \
	\prod_{j \in \NN} \pi_{s_j},
\end{equation}
(the operator  $I_\bs$ is defined in the same manner as $\Delta_\bs$).

For $\bs \in \FF$ and $\bk \in \pi_\bs$, let $E_\bs$ be the subset in $\FF$ of all $\be$ such that $e_j$ is either $1$ or $0$ if $s_j > 0$, and $e_j$ is $0$ if $s_j = 0$, and let $\by_{\bs;\bk}:= (y_{s_j;k_j})_{j \in \NN} \in \RRi$. Put $|\bs|_1 := \sum_{j \in \NN} s_j$ for  $\bs \in \FF$. It is easy to check that the interpolation operator $\Delta_\bs$ can be represented in the form
\begin{equation} \label{Delta_bs=}
	\Delta_\bs(v)				
	\ = \
	\sum_{\be \in E_\bs} (-1)^{|\be|_1} I_{\bs - \be} (v)
	\ = \
	\sum_{\be \in E_\bs} (-1)^{|\be|_1} \sum_{\bk \in \pi_{\bs - \be}} v(\by_{\bs - \be;\bk}) L_{\bs - \be;\bk}.
\end{equation}

For a given finite set $\Lambda \subset \FF$, we introduce the GPC interpolation operator $I_\Lambda$  by
\begin{equation*} 
	I_\Lambda
	:= \
	\sum_{\bs \in \Lambda} \Delta_\bs.
\end{equation*}
From \eqref{Delta_bs=} we obtain the sparse-grid interpolation sampling algorithm
\begin{equation} \label{I_Lambda=}
	I_\Lambda(v)				
	\ = \
	\sum_{\bs \in \Lambda} \sum_{\be \in E_\bs} (-1)^{|\be|_1} \sum_{\bk \in \pi_{\bs - \be}} v(\by_{\bs - \be;\bk}) L_{\bs - \be;\bk}.
\end{equation}

%

Next, we introduce  the multi-level sparse-grid interpolation  sampling algorithm  $\Ii_\xi$ for  $\xi >1$  by
\begin{equation}  \label{Ii_{xi}v}
	\Ii_{\xi} v 
	\ = \ 
	\sum_{k=0}^{k(\xi)} I_{\Lambda_k(\xi)}(\delta_k v).
\end{equation}
where 
\begin{equation} \nonumber
	k(\xi)
	:= \ 
	\begin{cases}
		\lfloor \log_2 \xi \rfloor \quad &{\rm if }  \ 
		\alpha \le 1/q_2 - 1/2;\\
		\lfloor \vartheta \eta^{-1}\log_2 \xi \rfloor \quad  & {\rm if }  \ 
		\alpha > 1/q_2 - 1/2,
	\end{cases}
\end{equation}
and
\begin{equation} \nonumber
	\Lambda_k(\xi)
	:= \ 
	\begin{cases}
		\big\{\bs \in \FF: \,\sigma_{2;\bs} \leq 2^{-k/q_2}\xi^{1/q_2 }\big\} \quad &{\rm if }  \ 
		\alpha \le 1/q_2 - 1/2;\\
		\big\{\bs \in \FF: \, \sigma_{1;\bs} \le \xi^{1/q_1}, \  
		\sigma_{2;\bs} \leq 2^{- \eta k}\xi^\vartheta \big\} \quad  & {\rm if }  \ 
		\alpha > 1/q_2 - 1/2,
	\end{cases}
\end{equation}
with 
\begin{equation} \label{eta}
	\eta:= \frac{2\alpha}{2 - q_2}, \qquad		
	\vartheta:= \frac{2}{2 - q_2} \left(\frac{1}{q_1} - \frac{1}{2}\right).
\end{equation}	
Since $\delta_k v$ belongs to the subspace $V_k + V_{k-1}$ of dimension at most $2^k + 2^{k-1}$ and the number of sampling points in the sampling algorithms $ I_{\Lambda_k(\xi)}$ is 
$$
n_k:= 	\sum_{\bs \in \Lambda} \sum_{\be \in E_\bs}
|\pi_{\bs - \be}|,
$$ 
the computational complexity 
$\operatorname{Comp} \brac{\Ii_{\xi}}$ of the operator  $\Ii_{\xi}$ can be defined as 
\begin{equation}  \nonumber
	\operatorname{Comp} \brac{\Ii_{\xi}}
	:= \ 
	\sum_{k=0}^{k(\xi)}  \brac{2^k + 2^{k-1}}n_k,
\end{equation}

\begin{theorem}\label{thm:fully-discrete-log-normal-holomorphic} 
	Let $v$ satisfy Assumption \ref{ass:ml} except the item 4.  	Let $q_i := p_i/(1-p_i)$ for $i=1,2$. 	Let	
	\begin{equation*} 	
		\beta := \brac{\frac 1 {q_1} - \frac{1}{2}}\frac{\alpha}{\alpha + \delta}, \quad 
		\delta := \frac 1 {q_1} - \frac 1 {q_2}.
	\end{equation*}		
Then there exists a constant $C$ such that for each $n \in \NN$, there exists a number $\xi_n$  such that  
	\begin{equation*} \label{comp1}
		\operatorname{Comp} \brac{\Ii_{\xi_n}}	
		\ \le  \
		n, 
	\end{equation*} 
	and
	\begin{equation*} 	
		\begin{split}
			\|v- \Ii_{\xi_n} v\|_{L_2(\RRi,X;\gamma)} 
			\ \le \	C
			\begin{cases}
				n^{-\alpha}\quad &{\rm if }  \ \alpha \le  1/q_2 - 1/2,\\
				n^{-\beta} \log n  \quad  & {\rm if }  \ \alpha > 1/q_2 - 1/2.
			\end{cases}
		\end{split}
	\end{equation*}
\end{theorem}	
\begin{proof}
	The operator $\Ii_{\xi}$ can be rewritten in the form
	\begin{equation} \label{Ii_xi= Ii_{G(xi)}}
		\Ii_\xi  \ = \ \Ii_{G(\xi)}  
		:= \
		\sum_{(k,\bs) \in G(\xi)} \Delta^{{\rm I}}_\bs \delta_k,
	\end{equation}
	where
	\begin{equation*} 
	G(\xi)
	:= \ 
	\begin{cases}
		\big\{(k,\bs) \in \NN_0 \times\FF: \, \sigma_{2;\bs}\leq \xi^{1/q_2}  2^{-k/q_2} \big\} 
		\quad &{\rm if }  \ \alpha \le 1/q_2 - 1/2;\\
		\big\{(k,\bs) \in \NN_0 \times\FF: \, \sigma_{1;\bs}\le \xi^{1/q_1} , \  
		\sigma_{2;\bs}\leq \xi^\vartheta 2^{-\eta k} \big\} 
		\quad  & {\rm if }  \ \alpha > 1/q_2 - 1/2.
	\end{cases}
\end{equation*}	
On the other hand, it holds the following statement which can be proven in a way similar to the proof of \cite[\black{Theorem 3.1}]{Dung2025-v14}.   Then there exists a constant $C$ such that for each $n \in \NN$, there exists a number $\xi_n$  such that  
\begin{equation*} \label{comp1}
	\operatorname{Comp} \brac{\Ii_{G(\xi_n)}}	
	\ \le  \
	n, 
\end{equation*} 
and
\begin{equation*} 	\label{L_2-rate}
	\begin{split}
		\|v- \Ii_{G(\xi_n)} v\|_{L_2(\RRi,X;\gamma)} 
		\ \le \	C
		\begin{cases}
			n^{-\alpha}\quad &{\rm if }  \ \alpha \le  1/q_2 - 1/2,\\
			n^{-\beta}  \log n \quad  & {\rm if }  \ \alpha > 1/q_2 - 1/2.
		\end{cases}
	\end{split}
\end{equation*}
From this statement and \eqref{Ii_xi= Ii_{G(xi)}} we prove the theorem.
	\hfill
\end{proof}

\section{Applications to  parametric elliptic PDEs} 
\label{Applications to  parametric PDEs} 
In this section, we prove that the parametric solution $u$ to
the parametric elliptic PDEs \eqref{ellip} on a bounded polygonal domain with log-normal inputs \eqref{lognormal} is a $(\bb_j,\xi,\delta,V)$-holomorphic function on $\RRi$ satisfying  Assumption \ref{ass:ml} with $\bb_j$ as in \eqref{eq:b1b2ml}, $j=1,2$. This allows to establish convergence rates of 
a fully discrete multi-level collocation  algorithm $S_n$.	The spatial components in $S_n$  are based on finite element Lagrange interpolations associated with triangulations of the spatial domain $D$. While   the parametric components  are based on Hermite-Lagrange GPC interpolation as in Theorem \ref{thm:sampling-HermiteInterpolation} in the case of small spatial regularity \black{$\alpha \le 1/p_2 - 1$}, and on  extended least squares sampling algorithms as in Theorem \ref{thm:sampling-log-normal-holomorphicK} in the cases of higher spatial regularity \black{$\alpha > 1/p_2 - 1$}.

It is convenient to consider the extended equation \eqref{ellip} with complex-valued data $a$ and $f$. We rewrite the solution to the equation \eqref{ellip} as the  mapping $a \mapsto \Uu(a)$ from $\Ww^{r-1}_\infty(D)$ to $\Kk_{\varkappa+1}^{r}(D)$ satisfying the equation
\begin{equation}\label{eq:ellipticK}
	- \div(a \nabla \Uu(a))=f\quad\text{in }D,\qquad
	\Uu(a) =0\quad\text{on }\partial D.
\end{equation}

In what follows, we assume that
\begin{equation*}
	\rho(a) := \underset{\bx\in D}{\operatorname{ess\,inf}}\, \Re (a(\bx)) > 0,
	\ \ \ \text{and} \ \ \
	|\varkappa|<\frac{\rho(a)}{\nu \norm{a}{L_\infty(D)}},
\end{equation*}
where $\nu$ is a constant depending on $D$ and $r$.  
The following result is proven in \cite[Theorem~7.8]{DNSZ2023}.
\begin{lemma}\label{lemma:bacuta2}
	Let $D\subset\RR^2$ be a bounded polygonal domain and
	$r\in\NN$, $r\ge 2$.  Then there exist $\varkappa>0$ and $C_r>0$
	depending on $D$ and $r$ such that for all
	$a\in W^{1}_\infty(D)\cap\Ww_{\infty}^{r-1}(D)$ and all
	$f\in \Kk_{\varkappa-1}^{r-2}(D)$ the weak solution
	$\Uu\in H_0^1(D)$ of \eqref{eq:ellipticK} belongs to the space $\Kk_{\varkappa+1}^{r}(D)$ and  satisfies with
	$N_r:=\frac{r(r-1)}{2}$
	\begin{equation*}
		\norm{\Uu}{\Kk_{\varkappa+1}^{r}(D)}
		\le C_r \frac{1}{\rho(a)}\left(\frac{\norm{a}{\Ww^{r-1}_{\infty}(D)}+\norm{a}{W^{1}_\infty(D)}}{\rho(a)}\right)^{N_r}\norm{f}{\Kk_{\varkappa-1}^{r-2}(D)}.
	\end{equation*}
\end{lemma}

We recall a concept of Lagrange finite element. Let $r \in \NN$ be given and $\Tt$  a triangulation of  $D$ with triangles $T$. Let $V(\Tt)$  be the finite element space which consists  of those continuous functions on $D$ that restrict to polynomials of order  at most $r$ on each triangle $T \in \Tt$.  For any function $v \in C(D)$, define $I v :=I(\Tt) v$ as the interpolating function associated to $v$, the interpolation nodes $\bx_1,...,\bx_{m(\Tt)} \in D$ being obtained by taking points with baricentric coordinates $r^{-1}\ZZ$. The operator $I$ is called Lagrange interpolation operator associated with $\Tt$. Thus, for any continuous function $v \in C(\bar{D)}$, the operator $I$ is uniquely determined by the conditions that 
$I v(\bx_i)= v(\bx_i)$ for any interpolation node $\bx_i$, $i=1,...,m(\Tt)$, and $I v \in V(\Tt)$. From the definitions we can see that
\begin{equation} \label{I v}
	(I v)(\bx): = \sum_{i=1}^{m(\Tt)}  v(\bx_i) \phi_{i}(\bx),
\end{equation}
where $\phi_{i}$, $i=1,...,m(\Tt)$, are the nodal basic of $V(\Tt)$. If $\Sigma:= \brab{\sigma_1,...,\sigma_{m(\Tt)}}$ are the linear forms such that $\sigma_i(v):= v(\bx_i)$, $i=1,...,m(\Tt)$, for $v \in V(\Tt)$,  then the triple $(D,V(\Tt),\Sigma)$ is a Lagrange finite element.

We consider a special class of the so-called shape-regular triangulation  $\Tt$ which satisfies the 
condition $\frac{R_{T}}{r_{T}}\le c$ for every $T \in \Tt$, 
where $c$ is a constant  and $r_{T}$ 
($R_{T}$) is the radius of the largest (smallest) ball contained in (containing) $T$. The following result is well-known (see, e.g., \cite{Brenner,Ciarlet}).

\begin{lemma}\label{lemma:InterpolationApprox}
	Let $r \in \NN$ and $I:= I(\Tt)$, $m:= m(\Tt)$.
Assume that  the triangulation  $\Tt$ is shape-regular.  Then there exists a constant $C = C_{r,a,c}$  such that  for any $v \in H^{r}(D)$
	\begin{equation*} 
		\norm{v- I v}{H^1(D)}
		\ \le \
	C m^{- \frac{r-1}{2}}	{\norm{v}{H^{r}(D)}}.	
	\end{equation*}
\end{lemma}

For the reader's convenience, we recall the definition of  the class $\Cc$  of   triangulations $\Tt$ of $D$ introduced  in \cite[Definition~4.1]{BNZ2005}. 
For simplicity of presentation we assume that $D$ is a triangle all of whose angles are acute. The general case can be considered in a similar way by first dividing $D$ in triangles with acute angles. We also assume that  $D$ is an open set. Let $l$ be the length of shortest edge of $D$. Let $D_\delta$ be the union of three isosceles triangles that have equal sides of length $\delta$. The complement $D \setminus D_\delta$ is a hexagon when $\delta < l$.
Fix $r \in \NN$ and let $\varkappa \in (0,1]$, $h >0$, $\epsilon, b \in (0,1)$, and $a \in (0,\pi/2)$ be parameters. We define $\Cc:= \Cc(r,h,\varkappa,\epsilon, a,b)$ to be the set of triangulations $\Tt$ defined as follows. Choose $m$ such that 
$$
\epsilon^{\varkappa m} 
\ \le \ M(l,\epsilon l /8,a) h^r.
$$
We decompose $D$ as the union of  $\Omega_0:= D \setminus D_{l/4}$, 
$\Omega_1:= D_{l/4} \setminus D_{\epsilon l/4}$,..., 
$\Omega_m:= D_{\epsilon^{m-1} l/4}\setminus D_{\epsilon^m l/4}$, and  
$\tilde{\Omega}_{m+1}:= D_{\epsilon^m l/4}$.
For each $j = 0,...,m$, we triangulate $\Omega_j$ with triangles  with all angles $\ge a$, and edges of length at most 
$$
h_{m,j}:= h \epsilon^{(1 - \varkappa/r)j}
$$
and at least $b h_{m,j}$. Then $\Tt$ is the union of the triangles appearing in the triangulations $\Omega_j$, $j \le m$, and of three triangles forming $\tilde{\Omega}_{m+1}$. 

Notice that the finite element space $V(\Tt)$ with $\Tt \in \Cc$, $r$ is the degree of the polynomials used in the approximation, $h$ is the largest admissible length of the sides of the triangles in the partition, $\epsilon$ controls the decay of the triangles as they approach a vertex, $a$ is the minimum admissible angle of a triangle in the partition , $0< \varkappa < \pi/a_1$, where $a_1$ is the largest angle of the polygon, and $b$ controls the ratio of the sizes of close triangles. The constants $r,h,\varkappa,\epsilon,a,b$ must satisfy certain conditions for the class $\Cc:= \Cc(r,h,\varkappa,\epsilon,a,b)$ to be non-empty. 
The following result \cite[Theorem 0.3]{BNZ2005} is therefore relevant.
For any polygon $D$, there exist $0 < \varkappa  \le 1$, $a  > 0$, $1 > b > 0$ and a sequence $h_k = h_0 2^{-k}$ such that, 
if $\epsilon = 2^{-r/\varkappa}$, the class $\Cc_k:= \Cc(r,h_k,\varkappa,\epsilon, a,b)$ is not empty, where $h_0 > 0$ is a certain constant.

If $\Tt_k \in \Cc_k$, there are positive constants $C$ and $C'$ such that 
\begin{equation*} 
	C n \le h_n^{-2} \le C' n, \ \ \dim V_k = m_k \le 2^{\black{k-1}}.
\end{equation*}	
where we denote $V_k:= V(\Tt_k)$ and $m_k:= m(\Tt_k)$. 

Let $\Tt_k \in \Cc_k$,  $k \in \NN$.   Let   $P_k :=I(\Tt_k)$ be  the Lagrange interpolation operator  defined as in~\eqref{I v}:
\begin{equation} \label{P_k}
	(P_k v)(\bx): = \sum_{i_k=1}^{m_k}  v(\bx_{k,i_k}) \phi_{k,i_k}(\bx),
\end{equation}
 where  $\bx_{k,1},...,\bx_{k,m_k} \in D$  are  interpolation nodes and $\phi_{k,1},...,\phi_{k,m_k} \in C(D)$  are the nodal basis of $V_k$ whose restriction to every $T \in \Tt_k$ is a polynomial of degree at most $r$.
From \cite[Theorem~4.4]{BNZ2005} one can derive the following

\begin{lemma}\label{lemma:FEM}
Let $r\in\NN$. Then there exists a constant $C>0$  such that 
 we have for every $k \in \NN$,
	\begin{equation*}
			\sup_{0\neq v\in \Kk_{\varkappa+1}^{r}(D)}
		\frac{\norm{v- P_k v}{\Kk^1_1(D)}}{\norm{v}{\Kk_{\varkappa+1}^{r}(D)}}
		\ \le \ C 2^{- \frac{r-1}{2} k}.
	\end{equation*}
\end{lemma}

\black{
From Lemma~\ref{lemma:FEM}  it follows that the operators 	$P_k$,  $k \in \NN$, are uniformly bounded in $V$.
We extend $P_k$ as bounded operators  in $L_2(\RRi,V;\gamma)$ (which with an abuse is denoted again by $P_k$) by the formula
$$
(P_k v)(\by):= P_k (v(\by)).
$$
}

Throughout the rest of this section,  we consider the equation \eqref{eq:ellipticK} with 
 the parametric diffusion coefficient \eqref{lognormal} satisfying  the condition
$\psi_j\in W^{1}_\infty(D)\cap\Ww_\infty^{r-1}(D)$, $j\in\NN$, and $f\in \Kk_{\varkappa-1}^{r-2}(D)$. Denote
\begin{equation}\label{eq:b1b2ml}
b_{1;j}:=\norm{\psi_j}{\black{L_\infty(D)}},\quad
b_{2;j}:=\max\big\{\norm{\psi_j}{W^{1}_\infty(D)},\norm{\psi_j}{\Ww_\infty^{r-1}(D)}\big\}
\end{equation}
and $\bb_1:=(b_{1;j})_{j\in\NN}$, $\bb_2:=(b_{2;j})_{j\in\NN}$.

%
%
%

\begin{lemma}\label{lemma:holomorphyKspace}
	Let $D\subset\RR^2$ be a bounded polygonal domain and
	$r\in\NN$, $r \ge  2$.  Let
	$\alpha = \frac{r-1}{2}$. Let 
	$\psi_k\in W^{1}_\infty(D)\cap\Ww_{\infty}^{r-1}(D)$, $k \in \NN$, and 
	$f\in \Kk_{\varkappa-1}^{r-2}(D)$. Let $0< p_1,  p_2 < 1$.
	For $i=1,2$, let the sequences $\bb_i:=(b_{i;j})_{j\in\NN}$ be defined as in \eqref{eq:b1b2ml} with
	$\bb_i\in\black{\ell_{p_i}}(\NN)$. 
Then there exist $\xi>0$ and $\delta>0$ such that for the parametric solution to \eqref{eq:ellipticK} with inputs \eqref{lognormal}
\begin{equation*}
	u(\by):=\Uu\Bigg(\exp\Bigg(\sum_{j\in\NN}y_j\psi_j\Bigg)\Bigg)
\end{equation*}
and for every $k\in\NN$,
\begin{enumerate}
	\item[{\rm (1)}]
	$u$ is
	$(\bb_1,\xi,\delta,V)$-holomorphic,
	\item[{\rm (2)}]
	$P_k u$ is
	$(\bb_1,\xi,\delta,V)$-holomorphic,
	\item[{\rm (3)}]
	 $u-P_k u$ is
	$(\bb_1,\xi,\delta,V)$-holomorphic,
	\item[{\rm (4)}]
	 $u-P_k u$ is
	$(\bb_2,\xi,\delta 2^{-\alpha k},V)$-holomorphic,
		\item[{\rm (5)}]
	For every $m \in \NN$,  the function 
	$w:=P_m u$ possesses the properties  {\rm (1)--(4)} with the parameter $\delta$ replaced by $C\delta$, the same parameters 
	$\alpha, \eta, p_1, p_2, \xi$ and sequences $\bb_1, \bb_2$.
\end{enumerate}
\end{lemma}
 \begin{proof}	The proof of this lemma is a modification of the proof of 
 	\cite[Proposition 7.12]{DNSZ2023}. We  apply \cite[Theorem 4.11]{DNSZ2023} on holomorphy of composite functions with
 	$E=\black{L_\infty(D)}$ and $X=V$ to prove the claims  {\rm (1)--(5)}. 
 	
 		\noindent		
	{\bf Step 1.} 	We verify the claim (1). By \cite[Proposition 7.12]{DNSZ2023}  there exist $\xi_1>0$ and $\delta_1>0$ such that $u(\by)$
	is $(\bb_1,\xi_1,\delta_1,V)$-holomorphic on the open set
		$$
		O_1=\set{a\in \black{L_\infty(D;\CC)}}{\rho(a)>0} \subset \black{L_\infty(D;\CC)}
		$$  
for some constants $\xi_1>0$ and $\delta_1>0$ depending on $O_1$. 
		
			\noindent		
	{\bf Step 2.} 
	Concerning the claim (2),
	by assumption, $b_{1;j}=\norm{\psi_j}{\black{L_\infty(D)}}$ satisfies
	$\bb_1=(b_{1;j})_{j\in\NN}\in \black{\ell_{p_1}(\NN )\subseteq \ell_1(\NN)}$, 
	which corresponds to the assumption (iv) of
	Theorem \cite[Theorem 4.11]{DNSZ2023}. It remains to verify  the
	assumptions (i), (ii) and
	(iii) of  \cite[Theorem 4.11]{DNSZ2023} for $P_k u$, $k \in \NN$.
	
\noindent
		(i) \  
		$P_k\Uu:O_1\to V$ is holomorphic since  the map
		$a\mapsto P_k\Uu(a)$ is a composition of holomorphic
		functions. 
		
\noindent		
		(ii)\  
		For all $a\in O_1$, we have  by Lemma \ref{lemma:InterpolationApprox} and 
		the well-known estimate 
		$\norm{\Uu(a)}{V} \le \frac{\norm{f}{V'}}{\rho(a)}$,
		\begin{equation} \label{norm2}
			\begin{split}		
			\norm{P_k\Uu(a)}{V}
			\ &\le \
			\norm{\Uu(a)}{V} + \norm{\Uu(a)- P_k\Uu(a)}{V}
			\\
			\ &\le \
			C\norm{\Uu(a)}{V}
		\	\le \ C \frac{\norm{f}{V'}}{\rho(a)}
			= C \frac{\norm{f}{H^{-1}(D)}}{\rho(a)}.				
			\end{split}
		\end{equation}
		
		\noindent		
		(iii)\  
	 For all $a$, $b\in O_1$ we have  by Lemma \ref{lemma:InterpolationApprox} and \cite[(4.21)]{DNSZ2023},
		\begin{equation} \label{norm}
			\norm{P_k\Uu(a)-P_k\Uu(b)}{V}
			\le 
			C \norm{\Uu(a)-\Uu(b)}{V}\le \norm{f}{H^{-1}(D)}
			\frac{1}{\min\{\rho(a),\rho(b)\}^2}
			\norm{a-b}{\black{L_\infty(D;\CC)}}.
		\end{equation}
	According to \cite[Theorem 4.11]{DNSZ2023} 
	the map
	$$\Uu \mapsto P_k\Uu\in  \black{L_2}(U,V;\gamma)$$ is
	$(\bb_1,\xi_2,\delta_2,V)$-holomorphic, for some fixed
	constants $\xi_2>0$ and $\delta_2>0$ depending on $O_1$ but	independent of $k$.
	
		\noindent		
	{\bf Step 3.}	 
	The claim (3) follows directly from Steps 1 and 2 that the
	difference $u - P_k u$ is $(\bb_1,\xi_3,\delta_3,V)$-holomorphic for some constants $\xi_3>0$ and $\delta_3>0$ depending on $O_1$ but	independent of $k$.
	
		\noindent		
	{\bf Step 4.} To show the claim (4), we set
	$$O_2=\set{a\in W^{1}_\infty(D)\cap\Ww^{s-1}_\infty(D)}{\rho(a)>0},$$ and verify again the
	assumptions (i), (ii) and
	(iii) of  \cite[Theorem 4.11]{DNSZ2023} with
	$E = W^{1}_\infty(D)\cap\Ww^{s-1}_\infty(D)$. First, observe
	that with
	$$
	b_{2;j}:=\max\big\{\norm{\psi_j}{\Ww^{s-1}_\infty(D)},\norm{\psi_j}{\black{W^{1}_\infty(D)}}\big\},
	$$
	by assumption
	$$
	\bb_2=(b_{2;j})_{j\in\NN}\in\black{\ell_{p_2}(\NN)\hookrightarrow \ell_1(\NN)},
	$$
	which corresponds to  the assumption (iv) of
	Theorem \cite[Theorem 4.11]{DNSZ2023}.
	
	\noindent
	(i) \  
	 For every $k\in\NN$, the mapping $\Uu- P_k\Uu:O_2\to V$ is holomorphic: Since $O_2$ can
		be considered a subset of $O_1$ (and $O_2$ is equipped with a
		stronger topology than $O_1$), Fr\'echet differentiability
		follows by Fr\'echet differentiability of
		$$\Uu- P_k\Uu: O_1\to V,$$ which holds by Step 3.
		
		\noindent
		(ii) \   For every $a\in O_2$, by 	the embedding inequality \eqref{EmbeddingInequality} and Lemmata  \ref{lemma:FEM} and \ref{lemma:bacuta2},
		\begin{equation*} 
			\norm{(\Uu- P_k\Uu)(a)}{V}
			\le  \delta_k
			\frac{\Big(\norm{a}{W^{1}_\infty(D)}+\norm{a}{\Ww^{r-1}_{\infty}(D)}\Big)^{N_r+1}}{\rho(a)^{N_r+2}},
		\end{equation*}	
		where $\delta_k:= K 2^{-\alpha k}  \norm{f}{\Kk_{\varkappa-1}^{r-2}\black{(D)}}$.
		
		\noindent
		(iii) \   For every $a$, $b\in O_2\subseteq O_1$, by \eqref{norm} and \cite[(4.21)]{DNSZ2023}, 
		\begin{equation} \label{norm1}
			\begin{aligned}
			\norm{(\Uu- P_k\Uu)(a)-(\Uu- P_k\Uu)(b)}{V} &\le
			\norm{\Uu(a)-\Uu(b)}{V}
			+\norm{P_k\Uu(a)- P_k\Uu(b)}{V}\nonumber\\
			&\le
			C \norm{f}{H^{-1}(D)}
			\frac{2}{\min\{\rho(a),\rho(b)\}^2}
			\norm{a-b}{\black{L_\infty(D;\CC)}}. 
		\end{aligned}
	\end{equation}	
		We conclude with \cite[Theorem 4.11]{DNSZ2023} 
		that there exist $\xi_4$ and $C_4$
		depending on $O_2$, $D$ but independent
		of $k$ such that $u-P_k u$ is
		$(\bb_2,\xi_4, C_4\delta_k,V)$-holomorphic.

	Summing up, the claims (1)--(4) hold with
	\begin{equation*}
		\xi:=\min\{\xi_1,\xi_2,\xi_3,\xi_4,\}\qquad\text{and}\qquad
		\delta:=\max\Big\{\tilde C_1, C_2, C_3, C_4 K\norm{f}{\Kk_{\varkappa-1}^{r-2}\black{(D)}}\Big\}.
	\end{equation*}
	
		\noindent		
		{\bf Step 5.} The claim (5) can be proven with the same arguments as ones in Steps 1--4 by using  the inequality  \eqref{norm2}:
		\begin{align*}
			\norm{(P_m\Uu)(a)}{V}
			\ \le \
			C \norm{\Uu(a)}{V}
		\end{align*}
		for every $m \in \NN$ and every $a \in O_2$,
	\hfill
\end{proof}

Let $\Tt_k \in \Cc_k$. Recall that we denote $V_k:= V(\Tt_k)$ and $m_k:= m(\Tt_k)$ and  
  $P_k :=I(\Tt_k)$   the interpolation operator defined as in \eqref{P_k}.
Notice that there are positive constants $C$ and $C'$ such that 
\begin{equation*}
	C n \le h_n^{-2} \le C' n, \ \ \dim V_k = m_k \le 2^{k-1}.
\end{equation*}	
 Let the operators $\delta_k$, $k \in \NN$, be defined in \eqref{delta_k} for  the sequence $(P_k)_{k\in\NN}$. Then we have 
 \begin{equation*} 
 	(\delta_k v)(\bx): = \sum_{i_k=1}^{m_k}  v(\bx_{k,i_k}) \phi_{k,i_k}(\bx) 
 	- \sum_{i_{k-1}=1}^{m_{k-1}}  v(\bx_{k-1,i_{k-1}}) \phi_{k-1,i_{k-1}}(\bx), \ \ v \in C(D),
 \end{equation*}
which can be rewritten as 
\begin{equation} \label{delta_k=}
	(\delta_k v)(\bx): = \sum_{i_k=1}^{\bar{m}_k}  v(\bar{\bx}_{k,i_k}) \bar{\phi}_{k,i_k}(\bx), \ \ v \in C(D),
\end{equation}
where $\bar{m}_k:= m_k + m_{k-1}$; 
$\bar{\bx}_{k,i_k} := \bx_{k,i_k}$, 
$\bar{\phi}_{k,i_k}:= \phi_{k,i_k}$  for $i_k= 1,..., m_k$, and 
$\bar{\bx}_{k,i_k} := \bx_{k-1,i_{k-1}}$,
 $\bar{\phi}_{k,i_k}:= - \phi_{k-1,i_{k-1}}$ for $i_k= m_k + 1,..., \bar{m}_k$. 
 
 Recall that for $n_k \in \NN$,  the operator
 \begin{equation*} 
 	S_{\black{n_k}}^V  v : = \sum_{j_k=1}^{\black{n_k}}  v\brac{\by_{k,j_k}} h_{k,j_k}
 \end{equation*}
 has been defined in \eqref{S_n_k} with $X=V$ for  $\by_{k,1},...,\by_{k,n_k} \in U$ and  $h_{k,1},...,h_{k,n_k} \in L_2(\RRi,\CC;\gamma)$.  
Then  for all $n \ge 2$, with $\delta_k$ as in \eqref{delta_k} the operator $\bar{\Ss}^V_n:= \Ss_{\lceil n/\log n \rceil}^V$ defined as in \eqref{Ss_{n}^X} and \eqref{S_n_k} can be rewritten as
\begin{equation}  \label{barSs_{n}^V}
	\bar{\Ss}^V_n v
	\ = \ 
	\sum_{k=1}^{k_{\lceil n/\log n \rceil}} 
	\sum_{i_k=1}^{\bar{m}_k} \sum_{j_k=1}^{\black{n_k}} v\brac{\bar{\bx}_{k,i_k},\by_{k,j_k}} \Phi_{k,i_k,j_k}(\bx,\by),
\end{equation}
where $\Phi_{k,i_k,j_k}(\bx,\by):= \bar{\phi}_{k,i_k}(\bx) h_{k,j_k}(\by)$ and $k_{\lceil n/\log n \rceil}$ is defined by \eqref{n_k,k_n} for  a sequence  $(\xi_{n})_{n \in \NN}$ of increasing positive numbers $\xi_n$ satisfying the condition \eqref{xi_n}.
The operator 
$\bar{\Ss}^V_n$ is a linear sampling algorithm in the space 
$L_2(\RRi,V; \gamma)$ defined for functions $v(\bx,\by)$ on  the spatial-parametric domain $D \times \RRi$ and based on the  sampling points 
$$
\operatorname{SamplePts}\brac{\bar{\Ss}^V_n}
\ =\
\brab{\big(\bx_{k,i_k}, \by_{k,j_k}\big): \ i_k=1,...,\bar{m}_k,\ j_k = 1,...,\black{n_k}, \ k = 0,..., k_{\lceil n/\log n \rceil}}.
$$

\begin{assumption} \label{assumptionK}
	$D\subset\RR^2$ is a bounded polygonal domain and
	$r\in\NN$, $r \ge  2$; $f\in \Kk_{\varkappa-1}^{r-2}(D)$ and
	$\psi_k\in W^{1}_\infty(D)\cap\Ww_{\infty}^{r-1}(D)$, $k \in \NN$. 
	For $i=1,2$, the sequences $\bb_i:=(b_{i,j})_{j\in\NN}$ defined as in \eqref{eq:b1b2ml} satisfy the condition 
	$\bb_i\in\black{\ell_{p_i}}(\NN)$ with $0< p_1\le  p_2 < 1$. 
\end{assumption}

\begin{theorem}\label{thm:sampling-Kspace} 
	Let Assumption \ref{assumptionK} hold.
	Let the numbers $\alpha$ and  $\beta$ be defined by
	\begin{equation} 	\label{alpha,delta, beta}
	\alpha:= \frac{r-1}{2}, \quad	\beta := \brac{\frac 1 {p_1} - \black{1}} \frac{\alpha}{\alpha + \delta}, \quad 
		\delta := \frac 1 {p_1} - \frac 1 {p_2}.
	\end{equation}	
\black{Let  $ \varepsilon$ be	an arbitrarily  positive number.}  Then for all $n \ge 2$ there exist  $\by_{k,1},...,\by_{k,n_k} \in \black{\RRi_0}$ and 
	$h_{k,1},...,h_{k,n_k} \in L_2(\RRi,\CC;\gamma)$, $k=0,...,k_{\lceil n/\log n \rceil}$,  such that for the linear sampling algorithm defined by \eqref{barSs_{n}^V}, we have that
	\begin{equation*} \label{comp1}
	\big|\operatorname{SamplePts}\brac{\bar{\Ss}^V_n}\big|
		\ \le  \
		n, 
	\end{equation*} 
	and for the parametric solution $u$ to the equation \eqref{parametricPDE} with log-normal random inputs,
	\begin{equation*} 
	\begin{split}
		\big\|u-\bar{\Ss}^V_n u\big\|_{L_2(\RRi,V;\gamma)} 
		\ \le \
		\begin{cases}
			C_\varepsilon	n^{-\alpha}(\log n )^{\alpha + \black{1/2}+\varepsilon}\quad &{\rm if }  \ \alpha < \black{1/p_2 - 1},\\
			C_\varepsilon	n^{-\alpha}(\log n )^{\alpha + \black{3/2} +\varepsilon}\quad &{\rm if }  \ \alpha = \black{1/p_2 - 1}, \\
			\black{C_\varepsilon}	n^{-\beta}
			(\log n )^{\black{1 + \beta (1 + 1/2 \alpha )+ \varepsilon\beta/\alpha}}\ \ 
			\quad  &{\rm if }  \ \alpha > \black{1/p_2 - 1},
		\end{cases}	
	\end{split}
\end{equation*}
	\black{where  the positive constants $C_\varepsilon$ depend  only on 
	$\alpha,  q_1,q_2, K,$ and, explicitly, on $\varepsilon$.}
	\end{theorem}	
	
	\begin{proof}	Under the hypothesis of this theorem, all the assumptions of Lemma \ref{lemma:holomorphyKspace} are satisfied. 	This yields that  Assumption \ref{ass:ml} holds for $u\in L_2(\RRi,V;\gamma)$. Hence, Theorem \ref{thm:sampling-log-normal-holomorphicK} is true for $u$. To complete the proof it is sufficient to notice that
		$$
		\big|\operatorname{SamplePts}\brac{\bar{\Ss}^V_n}\big|
		\ =\
	\operatorname{Comp} \brac{\bar{\Ss}^V_n}	
	\ \le  \
	n.
		$$
	\hfill
	\end{proof}
	
	 Let the operators $\delta_k$, $k \in \NN$, be defined in \eqref{delta_k} for  the sequence $(P_k)_{k\in\NN}$ given by \eqref{P_k}. Then 
 by the formulas \eqref{I_Lambda=} and \eqref{delta_k=} we  can  represent the operator $\Ii_{\xi}$ defined in \eqref{Ii_{xi}v}, as
 \begin{equation} \label{I_xi1}
 	\Ii_{\xi} v(\bx,\by) 
 	\ = \ 
 	\sum_{k=0}^{\lfloor \log_2 \xi \rfloor} \sum_{i_k=1}^{\bar{m}_k}  
 	\sum_{(\bs_k,\be_k,\bj_k) \in \Gamma_k(\xi)}  
 	(-1)^{|\be_k|_1} v(\bar{\bx}_{k,i_k},\by_{\bs_k - \be_k;\bj_k})
 	\bar{\phi}_{k,i_k}(\bx)L_{\bs_k - \be_k;\bj_k}(\by),
 \end{equation}
 where	
 \begin{equation*} 
 	\Gamma_k(\xi)				
 	:= \
 	\{(\bs_k,\be_k,\bj_k) \in \FF \times \FF \times \FF: 
 	\, \bs_k \in \Lambda_k(\xi), \ \be_k \in E_{\bs_k}, \ \bj_k \in \pi_{\bs_k - \be_k} \}.
 \end{equation*}
 We  rewrite  $\Ii_{\xi}$ in \eqref{I_xi1} in the form of sampling algorithm in $L_2(\RRi,V;\gamma)$
 \begin{equation*} 
 	\Ii_{\xi} v(\bx,\by) 
 	\ = \ 
 	\sum_{k=0}^{\lfloor \log_2 \xi \rfloor} \sum_{i_k=1}^{\bar{m}_k}  
 	\sum_{(\bs_k,\be_k,\bj_k) \in \Gamma_k(\xi)}  
 	v(\bar{\bx}_{k,i_k},\by_{\bs_k - \be_k;\bj_k})
 	\Phi_{i_k,\bs_k,\be_k,\bj_k}(\bx, \by)	,
 \end{equation*}
 where 
 $$
 \Phi_{i_k,\bs_k,\be_k,\bj_k}(\bx, \by)
 := 
 (-1)^{|\be_k|_1}\bar{\phi}_{k,i_k}(\bx)L_{\bs_k - \be_k;\bj_k}(\by).
 $$
 
 In a similar way to the proof of Theorem \ref{thm:sampling-Kspace},  from 
 Theorem \ref{thm:fully-discrete-log-normal-holomorphic} and Lemma \ref{lemma:holomorphyKspace} we derive

\begin{theorem}\label{thm:sampling-HermiteInterpolation} 
	Let Assumption \ref{assumptionK} hold.
	Let the numbers $\alpha$ and  $\beta$ be defined by
	\begin{equation} 	\label{alpha,delta, beta-H}
		\alpha:= \frac{r-1}{2}, \quad	\beta := \brac{\frac 1 {p_1} - \black{1}} \frac{\alpha}{\alpha + \delta}, \quad 
		\delta := \frac 1 {p_1} - \frac 1 {p_2}.
	\end{equation}	
	Then for each $n \in \NN$ there exists a number $\xi_n$  such that  
	\begin{equation} \label{comp1}
		\big|\operatorname{SamplePts}\Ii_{\xi_n}\big|
		\ \le  \
		n, 
	\end{equation} 
	and for the parametric solution $u$ to the equation \eqref{parametricPDE} with log-normal random inputs \eqref{lognormal}, 
	\begin{equation*} 	
		\begin{split}
			\big\|u - \Ii_{\xi_n} u\big\|_{L_2(\RRi,V;\gamma)} 
			\ \le 
			\	C
			\begin{cases}
				n^{-\alpha} \quad &{\rm if }  \ \alpha \le 1/p_2 - \black{3/2},\\
				n^{-\beta} \quad  & {\rm if }  \ \alpha > 1/p_2 - \black{3/2},
			\end{cases}
		\end{split}
	\end{equation*}
	where  the constant $C$ is independent of  $n$.
\end{theorem}	

By combining Theorems \ref{thm:sampling-Kspace} and \ref{thm:sampling-HermiteInterpolation} we obtain the following final results.

\begin{theorem}\label{thm:sampling-final} 
	Let Assumption \ref{assumptionK} hold. 	\black{Let  $ \varepsilon$ be	any fixed positive number.}  
	Let the numbers $\alpha$ and  $\beta$ be defined by \eqref{alpha,delta, beta}.
	Then for all $n \ge 2$ there exist  points $(\bx_1,\by_1),...,(\bx_n,\by_n) \in D \times \black{\RRi_0}$ and functions
	$\varphi_1,...,\varphi_n \in V$ and $h_1,...,h_n \in L_2(\RRi,\RR;\gamma)$  such that for the linear sampling algorithm $S_n$ on the spatial-parametric domain $D\times \RRi$ defined 
	by
	\begin{equation*}
		S_n(v)(\bx,\by): = \sum_{i=1}^n  v(\bx_i,\by_i) \varphi_i (\bx) h_i(\by),  
		\ \ \bx \in D, \ \ \by \in \RRi,
	\end{equation*}	
	and for the parametric solution $u$ to the equation \eqref{parametricPDE} with log-normal random inputs 	 \eqref{lognormal}, 
	it holds the error bounds
	 \begin{equation} 	\label{L_2-rate-S_n}
			\big\|u-S_n u\big\|_{L_2(\RRi,V;\gamma)} 
			\ \le \	
					\begin{cases}
				C	n^{-\alpha}\quad &{\rm if }  \ \alpha \le \black{1/p_2 - 3/2},\\
					C_\varepsilon	n^{-\alpha}(\log n )^{\alpha + \black{1/2}+\varepsilon}
				\quad &{\rm if }  \ \black{1/p_2 - 3/2 <  \alpha < 1/p_2 - 1},\\
				C_\varepsilon	n^{-\alpha}(\log n )^{\alpha + \black{3/2} +\varepsilon}\quad &{\rm if }  \ \alpha = \black{1/p_2 -1},\\
				\black{C_\varepsilon}	n^{-\beta}
				(\log n )^{\black{1 + \beta (1 + 1/2 \alpha )+ \varepsilon\beta/\alpha}} \quad  &{\rm if }  \ \alpha > \black{1/p_2 - 1},
			\end{cases}			
	\end{equation}
		\black{where  the positive constants $C_\varepsilon$ depend  only on 
		$\alpha,  q_1,q_2, K,$ and, explicitly, on $\varepsilon$.}
		\end{theorem}	

\black{
\begin{remark} \label{remark}
	{\rm
	We provide commentary on the related results in prior works and contrast them with Theorem~\ref{thm:sampling-final}, which presents the main findings of this paper. 
	Recall that the operators
	$(P_k)_{k\in\NN}$ are given by \eqref{P_k}  and associated with the Lagrange finite element spaces $(V_k)_{k\in\NN}$.
	
	We begin by comparing the convergence rates established in Theorem~\ref{thm:sampling-final} with the best-known convergence rates for fully discrete, multilevel Smolyak sparse-grid interpolation.
		Let $\boldsymbol{I}_{\varepsilon_{n}}^{\operatorname{ML}}$ be the multilevel sparse-grid interpolation algorithm defined in 
	\cite[Theorem 7.5]{DNSZ2023} of work at most $n$ which is in particular, based on the approximation properties of $(P_k)_{k\in\NN}$ as in Lemma~\ref{lemma:FEM}. Under Assumption \ref{assumptionK} with the additional condition $p_2 < 3/2$, for the parametric solution $u$  to the equation \eqref{parametricPDE} with log-normal random inputs given by \eqref{lognormal}, it has been recently established in  \cite[Theorem 7.5 and page 177]{DNSZ2023} that for any 
	$n \ge 2$,
	 	\begin{equation} 	\label{L_2-rate-ML}
	 	\|u-\boldsymbol{I}_{\varepsilon_{n}}^{\operatorname{ML}}u\|_{L_2(\RRi,V;\gamma)} 
	 	\ \leq \ C 
	 	\begin{cases}
	 		n^{-\alpha} \log n &\text{if } \alpha \leq  1/p_2 - 3/2,
	 		\\
	 		n^{-\beta'}  \log n &\text{if } \alpha >  1/p_2 - 3/2,
	 	\end{cases} 
	 \end{equation}
	where $\alpha$ is defined in \eqref{alpha,delta, beta} and
	\begin{equation*} 	
	\beta' := \brac{\frac 1 {p_1} - \frac 3 2} \frac{\alpha}{\alpha + \delta}, \quad 
		\delta := \frac 1 {p_1} - \frac 1 {p_2}.
	\end{equation*}
	Denote by {$a_n$ and $b_n$} the bounds (without constants)  
	for these convergence rates as in the right-hand sides of 
	 \black{\eqref{L_2-rate-ML} and \eqref{L_2-rate-S_n}}, respectively. 
	Evidently, ignoring values of constants,
	$a_n$ and $b_n$ cannot exceed $n^{-\alpha}$ 
	which is governed by the spatial regularity $\alpha$.
	By simple computation we can see that
	\begin{equation*} 
		\begin{cases}
			b_n = a_n (\log n)^{- 1}  \  &\text{if }  \ \alpha \leq  1/p_2 - 3/2,
			\\
			b_n = a_n n^{-\delta_1} (\log n )^{\alpha  - 1/2 +\varepsilon}  \  &
			\text{if }  \ 1/p_2 - 3/2< \alpha < 1/p_2 -\black{1},
			\\
			b_n = a_n n^{-\delta_1} (\log n )^{\alpha + 1/2 +\varepsilon}   \  &
			\text{if }  \ \alpha  = 1/p_2 -\black{1},
			\\
			b_n = a_n n^{-\delta_2}(\log n )^{\beta (1 + 1/2 \alpha )  + \varepsilon\beta/\alpha}     \ &
			\text{if } \ \alpha >  1/p_2 - \black{1},
		\end{cases}
	\end{equation*}	
	where
	$$
	\delta_1:= \alpha - \beta' >0 \ \ (\text{with} \ \alpha > 1/p_2 - 3/2),\qquad  
	\delta_2:= \frac{\alpha}{\alpha + \delta}>0.
	$$
	This shows that	the convergence rate $n^{-\alpha}$ in \eqref{L_2-rate-S_n}  improves the result in \eqref{L_2-rate-ML}
	 by a logarithm factor of $(\log n)^{-1}$ in the case of small spatial regularity $\alpha \le 1/p_2 - 3/2$.
	The convergence rates in \eqref{L_2-rate-S_n}  significantly improve the  convergence rates in \eqref{L_2-rate-ML} in the cases of higher spatial regularity $\alpha > 1/p_2 - 3/2$. Moreover, in \eqref{L_2-rate-S_n},  the maximal convergence rate $n^{-\alpha}$  still holds with a logarithm factor for  the particular intermediate case $1/p_2 - 3/2< \alpha \le 1/p_2 -1/2$.
	 
	 	We next compare the established convergence rates in Theorem~\ref{thm:sampling-final} with the best known convergence rates of best $n$-term (and optimal linear) fully discrete approximation applied to the Kondrat'ev analyticity context.	
	 Define for $\xi>1$,
	 \begin{equation*} 
	 	G(\xi)
	 	:= \ 
	 	\begin{cases}
	 		\big\{(k,\bs) \in \NN_0 \times\FF: \, 2^k \sigma_{2;\bs}^{q_2} \leq \xi\big\} \quad &{\rm if }  
	 		\ \alpha \le 1/p_2 - 1/2;\\
	 		\big\{(k,\bs) \in \NN_0 \times\FF: \, \sigma_{1;\bs}^{q_1} \le \xi, \  
	 		2^{\alpha q_1 k} \sigma_{2;\bs}^{q_1} \leq \xi\big\} \quad  & {\rm if }  \ \alpha > 1/p_2 - 1/2,
	 	\end{cases}
	 \end{equation*}
	 where $q_j:= 2p_j/(2 - p_j)$, $j=1,2$.
	Denote by $\Vv(G(\xi))$ the subspace in $L_2(\RRi,V;\gamma)$ of  all functions $v$
	of the form
	\begin{equation} \nonumber
		v
		\ = \
		\sum_{(k,\bs) \in G(\xi)} v_k \, H_\bs, \quad v_k \in V_k.
	\end{equation}
 We define the linear operator 
$
	\Ss_{G(\xi)}: \, L_2(\RRi,\Kk_{\varkappa+1}^{r}(D);\gamma)\to \Vv(G(\xi))
$ 
by the truncation of the Hermite GPC expansion
	\[
	\Ss_{G(\xi)} v
	:= \
	\sum_{(k,\bs) \in {G(\xi)}} \delta_k (v_\bs)  H_\bs
	\]
	for $v \in L_2(\RRi,\Kk_{\varkappa+1}^{r}(D);\gamma)$ represented by the series \eqref{GPCexpansion} with $X=V$.

Under Assumption \ref{assumptionK}, we consider the fully discrete (and multi-level) linear approximation by the operator $\Ss_{G(\xi_n)}$ of the parametric solution $u$  to the equation \eqref{parametricPDE} with log-normal random inputs given by \eqref{lognormal}. Building on the results of \cite[Theorem 2.1]{Dung21}, which establish fully discrete linear approximation in abstract Bochner spaces, together with Lemmata \ref{lemma:weighted summability, holomorphic} and \ref{lemma:bacuta2}--\ref{lemma:holomorphyKspace}, we obtain the following directly:
For each $n \in \NN$, there exists a number $\xi_n$  such that   $\dim(\Vv(G(\xi_n)) \le n$ and
		\begin{equation*} 	
			\|u-\Ss_{G(\xi_n)}u\|_{L_2(\RRi,V;\gamma)} \ \leq \ C 
			\begin{cases}
				n^{-\alpha} &\text{if } \alpha \leq  1/p_2 - 1/2,
				\\
				n^{-\beta}   &\text{if } \alpha >  1/p_2 - 1/2,
			\end{cases} 
		\end{equation*}
with the same $\alpha$ and  $\beta$ as  in Theorem~\ref{thm:sampling-final}. As noted in \cite[Remark 2.1]{Dung21} these convergence rates agree with the convergence rates of best $n$-term approximation (cf. also \cite[Theorem 3.1]{BCDC17}). 
Moreover, the convergence rates stated in Theorem~\ref{thm:sampling-final}	 coincide with them in the case when 
$\alpha \leq  1/p_2 - \black{3/2}$, and differ from them only by a logarithm factor  in the other cases.
}
\end{remark}
}

\section{Extensions of least squares sampling algorithms}
\label{Extensions}
In this section, we  discuss various  \black{choices of sample points  for least squares algorithms}  for functions in the reproducing kernel Hilbert space $H_{\CC,\bsigma}$, and inequalities between sampling $n$-widths and Kolmogorov $n$-widths of the unit ball $B_{\CC,\bsigma}$  of this space. We explain then how to apply these inequalities to obtain corresponding convergence rates of  multi-level linear sampling recovery in abstract Bochner spaces and  of fully discrete multi-level collocation  approximation of  the parametric solution $u$ to
the parametric elliptic PDEs \eqref{ellip} on a bounded polygonal domain $D$ with log-normal inputs \eqref{lognormal}.

Recall that in Subsection \ref{Extended least squares sampling algorithms in Bochner spaces}, a notion of weighted least squares sampling algorithm $S_n^\CC$ in $L_2(U,\CC;\mu)$ is  introduced as in \eqref{least-squares-sampling1}--\eqref{least-squares3}. Extensions of $S_n^\CC$ to $L_2(U,X;\mu)$ is defined as in \eqref{least-squares-sampling-extension}. By the help of extended  weighted least squares sampling algorithms  with the special choice of sample points $\by_1, \dots, \by_n$ and weights $\omega_1, \dots, \omega_n$, and the bounds of the approximation error as in Lemma \ref{lemma:sampling inequality}, we constructed efficient multi-level least squares sampling algorithms in the Bochner space  $L_2(U,X;\mu)$ for functions in $H^\alpha_{X}$, and proved the convergence rates by them as in Theorem~\ref{thm:sampling2}. 
The choice of sample points $\by_1, \dots, \by_n$, weights $\omega_1, \dots, \omega_n$, and approximation space $\Phi_m$ is crucial for the error of the least squares sampling algorithm.
We recall three choices presented in \cite{BD2024}, with a trade-off between constructiveness and tightness of the error bound. The third choice has been considered in Lemma \ref{lemma:sampling inequality}.

\begin{assumption}\label{assum:samplepoints}
	Let $n\in\NN$, $n\ge 90$, $c_1\ge 1$, $c_2 > 1+\frac{1}{n}$, and $c_3
	\ge 3284$. Let the probability measure $\nu$ be defined as in \eqref{nu}.
	\begin{itemize}
				\item[\rm (1)]
		Let $m := \lfloor n/(20\log n)\rfloor$.
		Let further $\by_1  ,\dots \by_{c_1n}\in U$ be points drawn i.i.d.\ with respect to $\nu$ and $\omega_i := (\varrho(\by_i))^{-1}$.
		\item[\rm (2)]
		Let $m := n$ and $\lceil 20 n\log n\rceil$ points be drawn i.i.d.\ with respect to $\nu$ and subsampled using \cite[Algorithm~3]{BSU23} to $c_2n \asymp m$ points.
		Denote the resulting points by $\by_1, \dots, \by_{c_2n} \in U$ and $\omega_i = \frac{c_2n}{\lceil 20 n\log n\rceil}(\varrho(\by_i))^{-1}$.
		\item[\rm (3)]
		Let $m := n$ and $\lceil 20 n\log n\rceil$ points be drawn i.i.d.\ with respect to $\nu$.
		Let further $\by_1, \dots, \by_{c_3n} \in U$ be the subset of points fulfilling \cite[Theorem~1]{DKU2023} with $c_3n \asymp m$ and $\omega_i := \frac{c_3n}{\lceil 20  n\log n\rceil}(\varrho(\by_i))^{-1}$.
	\end{itemize}
\end{assumption}

\begin{lemma}\label{lemma:leastsquaresbounds}
	Let $0 < p < 2$. Let  $d_n := d_n(B_{\CC,\bsigma},L_2(U,\CC;\mu))$. For $c,n,m\in\NN$ with $cn\ge m$, let $S_n^\CC$ be the least squares sampling algorithm  defined as in \eqref{least-squares-sampling1}--\eqref{least-squares3}.
	There are constants $c_1, c_2, c_3 \in \NN$ depending on $p$ such that for all  $n \ge 2$ we have the following.
	\begin{enumerate}
		\item[\rm (1)]
		The points from Assumption \ref{assum:samplepoints}{\rm (1)} fulfill with high probability
		\begin{equation} \nonumber
				\varrho_n(B_{\CC,\bsigma}, L_2(\RRi,\CC;\mu)) 
			\ \le \
			\sup_{v\in B_{\CC,\bsigma}} \norm{v - S_{c_1n}^X v}{L_2(U,\CC;\mu)}
				\ \le \brac{\frac{\log n}{n} \sum_{j \ge n/\log n} d_j^p}^{1/p}
		\end{equation}
		\item[\rm (2)]
		The points from Assumption~\ref{assum:samplepoints}{\rm (2)} fulfill with high probability
		\begin{equation}\nonumber
				\varrho_n(B_{\CC,\bsigma}, L_2(\RRi,\CC;\mu)) 
			\ \le \
			\sup_{v\in B_{\CC,\bsigma}} \norm{v - S_{c_2n}^X v}{L_2(U,\CC;\mu)}
				\ \le \brac{\frac{\log n}{n} \sum_{j \ge n} d_j^p}^{1/p}.
		\end{equation}
		\item[\rm (3)]
		The points from Assumption~\ref{assum:samplepoints}{\rm (3)} fulfill with high probability
		\begin{equation}\nonumber
				\varrho_n(B_{\CC,\bsigma}, L_2(\RRi,\CC;\mu)) 
			\ \le \
			\sup_{v\in B_{\CC,\bsigma}} \norm{v - S_{\black{c_3n}}^X v}{L_2(U,\CC;\mu)}
				\ \le \brac{\frac{1}{n} \sum_{j \ge n} d_j^p}^{1/p}.
		\end{equation}
	\end{enumerate}
\end{lemma}

This lemma has been formulated in \cite{BD2024}.
As mentioned there, the claims (1) and (3) have  been proven in 
\cite[Theorem~8]{KU21} and
\cite[Theorem~1]{DKU2023}, respectively.	
The claim (2) can be proven a similar way based on the result
\cite[Theorem~6.7]{BSU23}.

As commented in \cite{BD2024}, regarding the constructiveness of the linear least squares sampling algorithms in Lemma~\ref{lemma:leastsquaresbounds},  the bound  Lemma~\ref{lemma:leastsquaresbounds}(1) is the most coarse bound, but the points construction requires only a random draw, which is computationally inexpensive.
The sharper bound in Lemma~\ref{lemma:leastsquaresbounds}(2) uses an additional constructive subsampling step.
This was implemented and numerically tested in \cite{BSU23} for up to 1000 basis functions.
For larger problem sizes the current \black{algorithm is too slow} as its runtime is cubic in the number of basis functions.
The sharpest  bound in Lemma~\ref{lemma:leastsquaresbounds}(3) is a pure existence result.
So, up to now, the only way to obtain this point set is to brute-force every combination, which is computational infeasible.

As mentioned above, the linear least squares sampling algorithms in Lemma \ref{lemma:sampling inequality}(3) are based on the sample points in Assumption~\ref{assum:samplepoints}{\rm (3)} which give the best convergence rate among the sample points in Assumption~\ref{assum:samplepoints}{\rm (1)}--{\rm (3)}, but are least constructive. Hence,   the parametric components in the linear sampling algorithms in Theorems 
\ref{thm:sampling}, 
\ref{thm:sampling2}, 
\ref{thm:sampling-log-normal-holomorphicK} and
\ref{thm:sampling-Kspace}
are  based on such points. The linear sampling algorithms in Lemma~\ref{lemma:sampling inequality}(1)--(2)  based on the sample points in Assumption~\ref{assum:samplepoints}(1)--(2), respectively, are pure least squares algorithms or least squares algorithms with constructive subsampling, and therefore, constructive. But they give slightly worse  error bounds.

Again, let $(\xi_{n})_{n \in \NN}$ be a sequence of increasing positive numbers whose values will be selected later, such that $\xi_{n} \to \infty$ as $n \to \infty$. For $n \in \NN$ and $j =1,2,3$, we consider the multilevel least squares  operator $\Ss_{(j),n}^X$ defined by 
\begin{equation*}  
	\Ss_{(j),n}^X v
	\ = \ 
	\sum_{k=1}^{k_n}  S_{c_j n_{j,k}}^X (\delta_k v),
\end{equation*}
where  
\begin{equation*} 
	n_{j,k} := \lfloor  c_j^{-1} n 2^{-k}\rfloor, \ \ \ k_n:= \lfloor \log \xi_n \rfloor,
\end{equation*}
and
\begin{equation*} 
	S_{c_j n_{j,k}}^X  v : = \sum_{i=1}^{c_j n_k}  v\brac{\by_{k,i}^{(j)}} h_{k,i}^{(j)}
\end{equation*}
the extended least squares approximations with sample points
$\by_{k,1}^{(j)},...,\by_{k,n_k}^{(j)} \in U$  as in Assumption~\ref{assum:samplepoints}(j), $h_{k,1}^{(j)},...,h_{k,n_k}^{(j)} \in L_2(\RRi,\CC;\mu)$,
and $c_j > 0$ are the constants as in Lemma~\ref{lemma:leastsquaresbounds}, respectively. We define the operators:
\begin{equation}  \label{Ss_{(j),n}^X}
	\bar{\Ss}_{(j),n}^X 
	:= \ 
	\Ss_{(j),\lfloor n / \log n\rfloor}^X.
\end{equation}
Theorem \ref{thm:sampling2} gives the error bounds of approximation of  $v\in B_X^{\alpha,\tau}$ by the operators $\bar{\Ss}_{(3),n}^X = \bar{\Ss}_{n}^X$ based on the parametric sample points in Assumption~\ref{assum:samplepoints}(3) and proven with the help of Lemma~\ref{lemma:leastsquaresbounds}(3). Below we formulate its counterparts based on the parametric sample points in Assumption~\ref{assum:samplepoints}(1)--(2) and proven with the help of Lemma~\ref{lemma:leastsquaresbounds}(1)--(2), respectively, which can be proven in a similar way with slight modifications.

\begin{theorem}\label{thm:sampling-modification(1)} 
	Let the assumptions of Theorem \ref{thm:sampling2} hold.
	\black{Let  $\varepsilon$ be any fixed positive number.}
	Then for $j =1,2$ and all $n \ge 2$, there exist  $\by_{k,1},...,\by_{k,n_k} \in \black{\RRi_0}$ and $h_{k,1},...,h_{k,n_k} \in L_2(\RRi,\CC;\mu)$, $k=1,...,k_{\lceil n/\log n \rceil}$,  such that for the operator $\bar{\Ss}^X_{\black{(1)},n}$ defined as in \eqref{Ss_{(j),n}^X}, 
	\begin{equation*} \label{comp1}
		\operatorname{Comp} \brac{\bar{\Ss}^X_{(1),n}}	
		\ \le  \
		n,
	\end{equation*} 
	and
	\begin{equation*} 	\label{L_2-rate}
		\begin{split}
			\sup_{v\in B_X^{\alpha}}		\|v-\bar{\Ss}^X_{(1),n} v\|_{L_2(U,X;\mu)} 
			\ \le \	
			\begin{cases}
			C_\varepsilon	n^{-\alpha}\black{(\log n )^{\alpha + 1/2  +\varepsilon}}\quad &{\rm if }  \ \alpha < 1/q_2,\\
			C_\varepsilon	n^{-\alpha}(\log n )^{\alpha + 3/2 + 1/q_2 + \varepsilon}\quad &{\rm if }  \ \alpha = 1/q_2, \\
			\black{C_\varepsilon}	n^{-\beta}(\log n )^{\black{1 + \beta (2+ 1/2 \alpha ) + \varepsilon\beta/\alpha}}\ \ \quad  &{\rm if }  \ \alpha > 1/q_2.
			\end{cases}
		\end{split}
	\end{equation*}
\end{theorem}

\begin{theorem}\label{thm:sampling-modification(2)} 
	Let the assumptions of Theorem \ref{thm:sampling2} hold.
	\black{Let  $\varepsilon$ be any fixed positive number.}
	Then for $j =1,2$ and all $n \ge 2$, there exist  $\by_{k,1},...,\by_{k,n_k} \in \black{\RRi_0}$ and $h_{k,1},...,h_{k,n_k} \in L_2(\RRi,\CC;\mu)$, $k=1,...,k_{\lceil n/\log n \rceil}$,  such that for the operator $\bar{\Ss}^X_{(2),n}$ defined as in \eqref{Ss_{(j),n}^X}, 
	\begin{equation*} \label{comp1}
		\operatorname{Comp} \brac{\bar{\Ss}^X_{(2),n}}	
		\ \le  \
		n,
	\end{equation*} 
	and
	\begin{equation*} 	\label{L_2-rate}
		\begin{split}
			\sup_{v\in B_X^{\alpha}}		\|v-\bar{\Ss}^X_{(2),n} v\|_{L_2(U,X;\mu)} 
			\ \le \	
			\begin{cases}
			\black{C_\varepsilon}	n^{-\alpha}\black{(\log n )^{\alpha + 1/2  +\varepsilon}}\quad &{\rm if }  \ \alpha < 1/q_2,\\
			\black{C_\varepsilon}	n^{-\alpha}(\log n )^{\alpha + 2 + 2\varepsilon}\quad &{\rm if }  \ \alpha = 1/q_2, \\
			\black{C_\varepsilon}	n^{-\beta}(\log n )^{\black{3/2 + \beta (1+ 1/2 \alpha )+ \varepsilon(1 + \beta/\alpha)}} \  \quad  &{\rm if }  \ \alpha > 1/q_2.
			\end{cases}
		\end{split}
	\end{equation*}
\end{theorem}

From Theorems \ref{thm:sampling-modification(1)} and \ref{thm:sampling-modification(2)} one can derive respective results similar to Theorem \ref{thm:sampling-Kspace} and the others in Sections \ref{Applications to  holomorphic functions} and \ref{Applications to  parametric PDEs}, based on the least squares sampling algorithms defined as in \eqref{least-squares-sampling1}--\eqref{least-squares3} with the choice of sample points and weights as in Assumption \ref{assum:samplepoints}(1) or 
Assumption \ref{assum:samplepoints}(2).

 	
 \medskip
 \noindent
 {\bf Acknowledgments:}  
 This work is funded by the Vietnam National Foundation for Science and Technology Development (NAFOSTED)  under the Vietnamese--Swiss Joint Research Project, Grant No.  IZVSZ2$_{ - }$229568.  
 A part of this work was done when  the author was working at the Vietnam Institute for Advanced Study in Mathematics (VIASM). He would like to thank  the VIASM  for providing a fruitful research environment and working condition.
  
\bibliographystyle{abbrv}
\bibliography{Sampling&Coll-SPDE.bib}

\begin{thebibliography}{10}

\bibitem{ABDM2024}
B.~Adcock, S.~Brigiapaglia, N.Dexter, and S.~Moraga.
\newblock {Near-optimal learning of Banach-valued, high-dimensional functions
  via deep neural networks}.
\newblock {\em Neural Networks}, 181:106761, 2025.

\bibitem{ADM2024}
B.~Adcock, N.Dexter, and S.~Moraga.
\newblock {Optimal approximation of infinite-dimensional holomorphic functions
  II: recovery from i.i.d. pointwise samples}.
\newblock {\em J. Complexity}, 89:101933, 2025.

\bibitem{BNT2007}
I.~Babu{$\check{{\rm s}}$}ka, F.~Nobile, R.~Tempone, and C.~Webster.
\newblock {A stochastic collocation method for elliptic partial differential
  equations with random input data}.
\newblock {\em SIAM J. Num. Anal.}, 45:1005--1034, 2007.

\bibitem{BCDC17}
M.~Bachmayr, A.~Cohen, D.~{D\~ ung}, and C.~Schwab.
\newblock {Fully discrete approximation of parametric and stochatic elliptic
  PDEs}.
\newblock {\em SIAM J. Numer. Anal.}, 55:2151--2186, 2017.

\bibitem{BCDM17}
M.~Bachmayr, A.~Cohen, R.~DeVore, and G.~Migliorati.
\newblock {Sparse polynomial approximation of parametric elliptic PDEs. Part
  II: lognormal coefficients}.
\newblock {\em ESAIM Math. Model. Numer. Anal.}, 51:341--363, 2017.

\bibitem{BCM17}
M.~Bachmayr, A.~Cohen, and G.~Migliorati.
\newblock {Sparse polynomial approximation of parametric elliptic PDEs. Part I:
  affine coefficients}.
\newblock {\em ESAIM Math. Model. Numer. Anal.}, 51:321--339, 2017.

\bibitem{BD2024}
F.~Bartel and D.~D{\~u}ng.
\newblock {Sampling recovery in Bochner spaces and applications to parametric
  PDEs with log-normal random inputs}.
\newblock {\em arXiv e-preprint}, arXiv:2409.05050 [math.NA], 2024.

\bibitem{BSU23}
F.~Bartel, M.~Sch\"{a}fer, and T.~Ullrich.
\newblock Constructive subsampling of finite frames with applications in
  optimal function recovery.
\newblock {\em Applied and Computational Harmonic Analysis}, 65:209–248, July
  2023.

\bibitem{BTNT12}
J.~Beck, R.~Tempone, F.~Nobile, and L.~Tamellini.
\newblock {On the optimal polynomial approximation of stochastic PDEs by
  Galerkin and collocation methods}.
\newblock {\em Math. Models Methods Appl. Sci.}, 22:1250023, 2012.

\bibitem{Brenner}
B.~Brenner and L.~Scott.
\newblock {\em {The Mathematical Theory of Finite Element Methods}}, volume~15
  of {\em Texts in Applied Mathematics}.
\newblock Springer, New York, third edition, 2008.

\bibitem{BNZ2005}
C.~B\u{a}cu\c{t}\u{a}, V.~Nistor, and L.~T. Zikatanov.
\newblock Improving the rate of convergence of `high order finite elements' on
  polygons and domains with cusps.
\newblock {\em Numer. Math.}, 100(2):165--184, 2005.

\bibitem{CCMNT2015}
A.~Chkifa, A.~Cohen, G.~Migliorati, F.~Nobile, and R.~Tempone.
\newblock { Discrete least squares polynomial approximation with random
  evaluations application to parametric and stochastic elliptic PDEs}.
\newblock {\em ESAIM Math. Model. and Numer. Analysis}, 49:815--837, 2015.

\bibitem{CCS13}
A.~Chkifa, A.~Cohen, and C.~Schwab.
\newblock {High-dimensional adaptive sparse polynomial interpolation and
  applications to parametric PDEs}.
\newblock {\em Found. Comput. Math.}, 14(4):601--633, 2013.

\bibitem{CCS15}
A.~Chkifa, A.~Cohen, and C.~Schwab.
\newblock {Breaking the curse of dimensionality in sparse polynomial
  approximation of parametric PDEs}.
\newblock {\em J. Math. Pures Appl.}, 103:400--428., 2015.

\bibitem{Ciarlet}
P.~G. Ciarlet.
\newblock {\em {The Finite Element Method for Elliptic Problems}}.
\newblock North-Holland Publishing Co., Amsterdam-New York-Oxford, 1978.
\newblock Studies in Mathematics and its Applications, Vol. 4.

\bibitem{CoDe15a}
A.~Cohen and R.~DeVore.
\newblock {Approximation of high-dimensional parametric PDEs}.
\newblock {\em Acta Numer.}, 24:1--159, 2015.

\bibitem{CDS10}
A.~Cohen, R.~DeVore, and C.~Schwab.
\newblock {Convergence rates of best $N$-term Galerkin approximations for a
  class of elliptic sPDEs}.
\newblock {\em Found. Comput. Math.}, 9:615--646, 2010.

\bibitem{CDS11}
A.~Cohen, R.~DeVore, and C.~Schwab.
\newblock {Analytic regularity and polynomial approximation of parametric and
  stochastic elliptic PDE's}.
\newblock {\em Anal. Appl.}, 9:11--47, 2011.

\bibitem{Dung19}
D.~{D\~ung}.
\newblock {Linear collocation approximation for parametric and stochastic
  elliptic PDEs}.
\newblock {\em Mat. Sb.}, 210:103--227, 2019.

\bibitem{Dung21}
D.~{D\~ung}.
\newblock {Sparse-grid polynomial interpolation approximation and integration
  for parametric and stochastic elliptic PDEs with lognormal inputs}.
\newblock {\em ESAIM Math. Model. Numer. Anal.}, 55:1163--1198, 2021.

\bibitem{DD-Erratum23}
D.~{D\~ung}.
\newblock {Erratum to: ``Sparse-grid polynomial interpolation approximation and
  integration for parametric and stochastic elliptic PDEs with lognormal
  inputs", [Erratum to: ESAIM: M2AN 55(2021) 1163--1198]}.
\newblock {\em ESAIM Math. Model. Numer. Anal.}, 57:893--897, 2023.

\bibitem{Dung22}
D.~{D\~ung}.
\newblock {Collocation approximation by deep neural ReLU networks for
  parametric elliptic PDEs with lognormal inputs}.
\newblock {\em Mat. Sb.}, 214:38--75, 2023 (see also English version
  arXiv:2111.05504 [math.NA]).

\bibitem{Dung2025-v14}
D.~{D\~ung}.
\newblock {Sparse-grid polynomial interpolation approximation and integration
  for parametric and stochastic elliptic PDEs with lognormal inputs}.
\newblock {\em https://arxiv.org/abs/1904.06502v14}, 2025.

\bibitem{DKU2023}
M.~Dolbeault, D.~Krieg, and M.~Ullrich.
\newblock {A sharp upper bound for sampling numbers in $L_2$}.
\newblock {\em Appl. Comput. Harmon. Anal.}, 63:113--134, 2023.

\bibitem{DNSZ2023}
D.~D{\~u}ng, V.~Nguyen, C.~Schwab, and J.~Zech.
\newblock {\em {Analyticity and Sparsity in Uncertainty Quantification for PDEs
  with Gaussian Random Field Inputs}}.
\newblock Lecture Notes in Mathematics vol. 2334, Springer, 2023.

\bibitem{EST18}
O.~G. Ernst, B.~Sprungk, and L.~Tamellini.
\newblock Convergence of sparse collocation for functions of countably many
  {G}aussian random variables (with application to elliptic {PDE}s).
\newblock {\em SIAM J. Numer. Anal.}, 56:877--905, 2018.

\bibitem{HoSc14}
V.~Hoang and C.~Schwab.
\newblock {$N$-term Galerkin Wiener chaos approximation rates for elliptic PDEs
  with lognormal Gaussian random inputs}.
\newblock {\em Math. Models Methods Appl. Sci.}, 24:797--826, 2014.

\bibitem{KUV21}
L.~K\"{a}mmerer, T.~Ullrich, and T.~Volkmer.
\newblock Worst-case recovery guarantees for least squares approximation using
  random samples.
\newblock {\em Constructive Approximation}, 54(2):295–352, Aug. 2021.

\bibitem{KU21a}
D.~Krieg and M.~Ullrich.
\newblock {Function values are enough for $L_2$-approximation}.
\newblock {\em Found. of Comput. Math.}, 21(4):1141--1151, Oct. 2021.

\bibitem{KU21}
D.~Krieg and M.~Ullrich.
\newblock {Function values are enough for $L_2$-approximation: Part II}.
\newblock {\em Journal of Complexity}, 66:101569, Oct. 2021.

\bibitem{MNST2014}
G.~Migliorati, F.~Nobile, E.~von Schwerin, and R.~Tempone.
\newblock {Analysis of discrete $L^2$ projection on polynomial spaces with
  random evaluations}.
\newblock {\em Found. Comput. Math.}, 14:419--456, 2014.

\bibitem{NTW2008}
F.~Nobile, R.~Tempone, and C.~Webster.
\newblock {A sparse grid stochastic collocation method for elliptic partial
  differential equations with random input data}.
\newblock {\em SIAM J. Num. Anal.}, 46:2309--2345, 2008.

\bibitem{NTW2008a}
F.~Nobile, R.~Tempone, and C.~Webster.
\newblock {An anisotropic sparse grid stochastic collocation method for
  elliptic partial differential equations with random input data}.
\newblock {\em SIAM J. Num. Anal.}, 46:2411--2442, 2008.

\bibitem{ZDS19}
J.~Zech, D.~{D\~ung}, and C.~Schwab.
\newblock {Multilevel approximation of parametric and stochastic PDES}.
\newblock {\em {Math. Models Methods Appl. Sci.}}, 29:1753--1817, 2019.

\bibitem{ZS20}
J.~Zech and C.~Schwab.
\newblock {Convergence rates of high dimensional Smolyak quadrature}.
\newblock {\em ESAIM Math. Model. Numer. Anal.}, 54:1259--307, 2020.

\end{thebibliography}
\end{document}